\documentclass[11pt]{amsart}
\usepackage{amsfonts} % za izpis oznak za mnozice stevil
\usepackage{amsmath} % za definicijo matematicnih operatorjev, kot je npr. arctg, in nekaterih simbolov
\usepackage{amsthm} % izpise piko za stevilko definicije, ce je definicija vpeljana kot newtheorem
\usepackage{amssymb} % za razne simbole
\usepackage{enumitem} % da lahko nadaljujes enumerate
\usepackage{mathrsfs} % zakrivljena pisava
\usepackage{xcolor}

\newcommand{\set}[2]{\{#1 \ : \ #2\}}

\DeclareMathOperator{\End}{End}
\DeclareMathOperator{\ann}{ann}
\DeclareMathOperator{\im}{im}
\DeclareMathOperator{\cl}{cl}

\newtheorem{theorem}{Theorem}[section]
\newtheorem{proposition}[theorem]{Proposition}
\newtheorem{lemma}[theorem]{Lemma}

\theoremstyle{definition}
\newtheorem{definition}[theorem]{Definition}

\newtheorem{remark}[theorem]{Remark}
\newtheorem{question}[theorem]{Question}

\newenvironment{step}[1]{\par\smallskip\noindent\textbf{Step~#1.}}{\par\smallskip}

\begin{document}

\title[]{Compressed zero-divisor graphs of noncommutative rings}

\author{Alen \DJ uri\' c}
\address[A. \DJ uri\' c]{Faculty of Natural Sciences and Mathematics, University of Banja Luka, Mladena Stojanovi\' ca 2, 78000 Banja Luka, Bosnia and Herzegovina}
\email{alen.djuric@protonmail.com}

\author{Sara Jev\dj eni\' c}
\address[S. Jev\dj eni\' c]{Faculty of Natural Sciences and Mathematics, University of Banja Luka, Mladena Stojanovi\' ca 2, 78000 Banja Luka, Bosnia and Herzegovina}
\email{sarajevdjenic9@gmail.com}

\author{Nik Stopar}
\address[N. Stopar]{Faculty of Electrical Engineering, University of Ljubljana, Tr\v za\v ska cesta 25, 1000 Ljubljana, Slovenia}
\email{nik.stopar@fe.uni-lj.si}

\begin{abstract}
We extend the notion of the compressed zero-divisor graph $\varTheta(R)$ to noncommutative rings in a way that still induces a product preserving functor $\varTheta$ from the category of finite unital rings to the category of directed graphs.
For a finite field $F$, we investigate the properties of $\varTheta(M_n(F))$, the graph of the matrix ring over $F$, and give a purely graph-theoretic characterization of this graph when $n \neq 3$.
For $n \neq 2$ we prove that every graph automorphism of $\varTheta(M_n(F))$ is induced by a ring automorphism of $M_n(F)$.
We also show that for finite unital rings $R$ and $S$, where $S$ is semisimple and has no homomorphic image isomorphic to a field, if $\varTheta(R) \cong \varTheta(S)$, then $R \cong S$.
In particular, this holds if $S=M_n(F)$ with $n \neq 1$.
\end{abstract}

\maketitle

{\footnotesize \emph{Key Words:} zero-divisor, compressed zero-divisor graph, matrix ring

}

\section{Introduction}

The zero-divisor graph of a unital commutative ring $K$ was first introduced in 1988 by Beck \cite{Bec} as a tool to study the structure of the ring. His graph $G(K)$ is a simple graph with vertex set $K$ where two distinct vertices are adjacent if and only if their product is equal to $0$. Beck studied the chromatic number and the clique number of this graph.
Later Anderson and Livingston \cite{And-Liv} modified Beck's definition and defined a simple graph $\Gamma(K)$ by taking only nonzero zero-divisors of $K$ as the vertex set and leaving the definition of edges the same.
The properties of $\Gamma(K)$ were subsequently studied by several authors \cite{Akb-Mai-Yas, Akb-Moh, And-etal, And-Liv, Smi}. In particular, the isomorphism problem for such graphs has been considered in \cite{And-etal} and solved for finite reduced rings.
For more results on these graphs we refer the reader to a survey paper \cite{And-Axt-Sti}.

To further reduce the size of the graph, Mulay \cite{Mul} introduced the graph of equivalence classes of zero-divisors $\Gamma_E(K)$, which was later called compressed zero-divisor graph by Anderson and LaGrange \cite{And-LaG}.
The vertex set of $\Gamma_E(K)$ is the set of all equivalence classes of nonzero zero-divisors of $K$, where two zero-divisors are equivalent if and only they have the same annihilator in $K$.
Two equivalence classes $[x]$ and $[y]$ are adjacent in $\Gamma_E(K)$ if and only if $xy=0$. Thus $\Gamma_E(K)$ is obtained from $\Gamma(K)$ by compressing the vertices of $\Gamma(K)$ with the same neighbourhood in $\Gamma(K)$ into one vertex. Compressed zero-divisor graphs were extensively studied in \cite{And-LaG, And-LaG-2, Coy-etal, Spi-Wic}.

The advantage of $\Gamma_E(K)$ is that it may be much smaller than $\Gamma(K)$. In particular, $\Gamma_E(K)$ can be a finite graph even if $K$ is an infinite ring and $\Gamma(K)$ an infinite graph.
However, the disadvantage of this type of compression is that $\Gamma_E$ cannot be extended in a natural way to a functor from the category of  commutative unital rings to the category of graphs. There are two reasons for this.
Firstly, only zero-divisors are included in the definition of $\Gamma_E(K)$, however a homomorphic image of a zero-divisor can easily be invertible.
And secondly, the graph is compressed too much, since the homomorphic image of two equivalent zero-divisors need not be equivalent in general. 
To overcome this shortcoming, a new type of compressed zero-divisor graph $\varTheta(K)$ of a commutative unital ring was recently introduced by the authors in \cite{Dju-Jev-Sto}. We recall the definition of $\varTheta(K)$ for finite rings, the definition for infinite rings can be found in \cite[Definition~6.1]{Dju-Jev-Sto}.
The vertex set of $\varTheta(K)$ is the set of associatedness classes of all elements of $K$. Recall that two elements $a,b \in K$ are associated, denoted $a \sim b$, if $a=bu$ for some invertible element $u \in K$. Two associatedness classes $[a]_\sim$ and $[b]_\sim$ (not necessarily distinct) are adjacent in $\varTheta(K)$ if and only if $ab=0$.
It was proved in \cite{Dju-Jev-Sto} that this construction extends to a product preserving functor $\varTheta$ from the category of finite commutative unital rings to the category of graphs, and that the compression by associatedness relation is the optimal for this purpose. The graph structure of $\varTheta(K)$ was used to characterize important classes of finite commutative rings, namely, local rings and principal ideal rings.

There are also several other graphs constructions over commutative rings related to the zero-divisor graph that appear in the literature, see \cite{Beh-Bey, Dju-Jev-Obl-Sto, Red-3}.

In 2002 Redmond \cite{Red} extended the concept of zero-divisor graphs to noncommutative rings. For a general unital ring $R$ he defined a simple directed graph $\Gamma(R)$ and a simple undirected graph $\overline{\Gamma}(R)$ as follows.
The vertex set of both graphs is the set of all nonzero zero-divisors of $R$. In $\Gamma(R)$ there is a directed edge $x \to y$ if and only if $x \neq y$ and $xy=0$, while in $\overline{\Gamma}(R)$ there is an undirected edge $x - y$ if and only if $x \neq y$ and either $xy=0$ or $yx=0$.
Properties of these graphs have been studied in \cite{Akb-Moh-2, Akb-Moh-3, Boz-Pet, Ma-Wan-Zho, Red-2}. Most attention has been devoted to matrix rings either over fields or over general commutative unital rings, but other ring such as group rings, polynomial rings etc. have been considered as well.

Recently, Ma, Wang and Zhou \cite[Definition~2.5]{Ma-Wan-Zho} extended the notion of compressed zero-divisor graph $\Gamma_E(K)$ to noncommutative rings.
For a noncommutative ring $R$ graph $\Gamma_E(R)$ is a compression of $\Gamma(R)$ by an equivalence relation where two zero-divisors are equivalent if and only if they have the same left annihilator and the same right annihilator in $R$.
They described the connection between the automorphism groups of $\Gamma_E(M)$ and $\Gamma(M)$ where $M$ is the ring of $2 \times 2$ matrices over a finite field.

In this paper we extend the notion of the compressed zero-divisor graph $\varTheta(K)$ to noncommutative unital rings (see Definitions~\ref{def:theta} for finite rings and Definition~\ref{def:inf} for infinite rings) in order to obtain a functor $\varTheta$ from the category of finite unital rings to the category of directed graphs.
All the relevant definitions are contained in Section~\ref{sec:cat}, where we also show that functor $\varTheta$ enjoys similar properties as in the commutative case, namely, it preserves finite products in both directions (see Propositions~\ref{prop:noncommutative--preserves--products} and \ref{prop:noncommutative--preproduct}).

In Section~\ref{sec:matrices} we discuss the properties of the compressed zero-divisor graph $\varTheta(M_n(F))$ of the ring of $n \times n$ matrices over a finite field $F$. We show that it can be described in terms of pairs of vector subspaces of $F^n$, which makes it easier to work with and allows us to bring projective geometry into play later on.
We determine the number of vertices of $\varTheta(M_n(F))$, the sizes of their neighbourhoods, the existence of Hamiltonian paths and cycles in $\varTheta(M_n(F))$, the size of largest directed cliques, and the size of smallest dominating sets.
In Section~\ref{sec:characterization} we give a purely graph-theoretic characterization of graph $\varTheta(M_n(F))$ for $n\neq 3$, with no reference to the ring structure whatsoever (see Theorem~\ref{prop:noncommutative--staircase}). To the best of our knowledge this is the first result of this kind for matrix rings.

In Section~\ref{sec:structure} we study two standard questions; (1) Does $\varTheta(R) \cong \varTheta(S)$ imply $R \cong S$? and (2) Is any graph automorphism of $\varTheta(R)$ induced by a ring automorphism of $R$?
We answer question (1) positively when $S=M_n(F)$, a matrix ring over a finite field $F$, with $n \neq 1$, and more generally when $S$ is a semisimple finite ring with no homomorphic image isomorphic to a field (see Theorems~\ref{prop:matrix--same--graph--same--ring} and \ref{thm:iso-semisimple}).
We also give a positive answer to question (2) when $R=M_n(F)$ with $n \neq 2$ (see Theorem~\ref{thm:automorphisms}).
Lastly, in Section~\ref{sec:inf} we discuss the definition of $\varTheta(R)$ for infinite rings.

\section{Definition and categorial properties of $\varTheta(R)$}\label{sec:cat}

Throughout the paper, $F$ will denote a general finite field, and for a prime power $q>1$, $F_{q}$ will denote the unique finite field with $q$ elements. Unless specified otherwise, all rings will be finite unital rings, except in the last section, where we discuss the definition of $\varTheta(R)$ for infinite rings.
In this section we will extend some of the categorial results obtained in \cite{Dju-Jev-Sto} to noncommutative rings. We first need an appropriate equivalence relation to compress the zero divisor graph $\Gamma(R)$. We remark that our graphs will primarily be directed graphs.

Let $\sim$ denote a relation on a finite unital ring $R$, defined
by $a\sim b$ if and only if $a=bu=vb$ for some units $u$ and $v$.
Observe that on a commutative ring, the relation $\sim$ is just the associatedness relation, so there should be no confusion if we use the same notation for the relation as was used in the commutative setting in \cite{Dju-Jev-Sto}.
We claim that relation $\sim$ is an equivalence relation. Reflexivity is obvious.
The relation is symmetric since the equality $x=yu=vy$, where $u$ and $v$ are units, implies $y=xu^{-1}=v^{-1}x$. If $x\sim y$ and $y\sim z$, then $x=yu=vy$ and
$y=zs=tz$ for some units $u,v,s,t$, hence $x=z\left(su\right)=\left(vt\right)z$. This implies $x\sim z$, so the relation is also transitive.
The equivalence class of an element $a$ with respect to $\sim$ will be denoted simply by $[a]$, while the equivalence class with respect to any other equivalence relation $\approx$ will be denoted by $[a]_\approx$.

The following proposition shows that, at least on matrix rings over fields, the relation $\sim$ is the best possible equivalence relation to compress the zero-divisor graph because it will compress the graph as much as possible in a way that edges will be well-defined by zero products.

\begin{proposition}\label{prop:noncommutative--equivalence}
Let $\approx$ be an equivalence
relation on the ring $R=M_{n}\left(F\right)$ such that for all $x,y,z,w\in R$,
$x\approx z$, $y\approx w$ and $xy=0$ imply $zw=0$. Then for
any $a,b\in R$, $a\approx b$ implies $a\sim b$.
\end{proposition}

\begin{proof}
Let $a\approx b$. Since the ring $M_{n}\left(F\right)$ is unit-regular,
element $a$ can be written as $a=es$, where $e^{2}=e$ and $s$
is a unit. Then $1-e\approx1-e$, $a\approx b$ and $\left(1-e\right)a=0$
imply $\left(1-e\right)b=0$, so $b=eb=as^{-1}b\in aR$. Since $\approx$
is symmetric, we also have $a\in bR$, hence $aR=bR$. This means
that matrices $a$ and $b$ have the same image, so there is an invertible
matrix $u$ such that $a=bu$. Similarly, there is an invertible matrix
$v$, such that $a=vb$.
\end{proof}

Next we show that if we want to obtain a functor, the relation $\sim$ is the best possible also on products of finite unital algebras over fields. This partially generalizes \cite[Proposition~3.1]{Dju-Jev-Sto}. Before we can state the result, we need the following auxiliary proposition. We denote by $\ann_\ell a$ and $\ann_r a$ the left and right annihilator of $a$ respectively.

\begin{proposition}\label{prop:aux}
Let $R$ be a finite unital ring and $a, b \in R$. Then $aR=bR$ if and only if $a=bu$ for some invertible element $u \in R$.
\end{proposition}

\proof The if part of the claim is obvious. To prove the only if part, assume that $aR=bR$. Choose $x,y \in R$ such that $a=bx$ and $b=ay$.
Then $b(1-xy)=0$, so $1-xy$ is an element of $\ann_r b$, which is a right ideal of $R$. This implies that $1 \in xR+\ann_r b$, hence $xR+\ann_r b=R$.
Every finite unital ring is (left) artinian and hence semilocal, so by Bass' Theorem (see \cite[Theorem~20.9]{Lam}), the coset $x+\ann_r b$ contains an invertible element of $R$, say $u=x+z$, where $z \in \ann_r b$. Multiplying from the left by $b$ we obtain $bu=bx=a$. \endproof

In what follows, the category of finite unital rings and unital ring homomorphisms
will be denoted by $\mathbf{FinRing}$ and the category of directed
graphs and graph homomorphisms will be denoted by $\mathbf{Digraph}$.

\begin{proposition}\label{prop:best_possible}
For every $R\in\mathrm{obj}\mathbf{FinRing}$ let $\approx_{R}$ be
an equivalence relation on $R$, such that the family $\left\{ \approx_{R}\right\} _{R\in\mathrm{obj}\mathbf{FinRing}}$
induces a well-defined functor $T:\mathbf{FinRing}\to\mathbf{Digraph}$
in the following way.
\begin{enumerate}
\item \label{enu:compression-1} For $R\in\mathrm{obj}\mathbf{FinRing}$, the
vertices of $T\left(R\right)$ are equivalence classes of $\approx_{R}$,
and there is an edge $\left[a\right]_{\approx_{R}}\rightarrow\left[b\right]_{\approx_{R}}$
if and only if $ab=0$.
\item\label{enu:arrow-1} For $f\in\mathbf{FinRing}$$\left(R,S\right)$,
we have $T\left(f\right)\left(\left[a\right]_{\approx_{R}}\right)=\left[f\left(a\right)\right]_{\approx_{S}}$.
\end{enumerate}
If $A\in\mathrm{obj}\mathbf{FinRing}$ is a direct product of finite
unital algebras over possibly different fields, then $a\approx_{A}b$
implies $a\sim b$.
\end{proposition}

\begin{proof}
First, suppose $A$ is an algebra over a field $F$. Let $\End_{F}\left(A\right)$
denote the algebra of all $F$-linear maps $A \to A$. Let $L:A \to \End_{F}\left(A\right)$
be the left regular representation of $A$, i.e. $L\left(r\right)=L_{r}$,
where $L_{r}$ denotes left multiplication by $r$. Suppose $a\approx_{A}b$.
Then $L_{a}\approx_{\End_{F}\left(A\right)}L_{b}$ by \ref{enu:arrow-1}. Since
$A$ is finite dimensional vector space over $F$, the algebra $\End_{F}\left(A\right)$
is isomorphic to a matrix algebra. By \ref{enu:compression-1} the conditions of Proposition~\ref{prop:noncommutative--equivalence} are satisfied. Using \ref{enu:arrow-1}, we thus conclude that $L_{a}\sim L_{b}$, i.e. $L_{a}=L_{b}\circ U=V\circ L_{b}$
for some invertible $U,V\in\End_{F}\left(A\right)$. Applying
these maps to $1\in A$, we get $a=b\cdot U\left(1\right)=V\left(b\right)$.
This implies $a\in bA$. Since $\approx_{A}$ is symmetric, we also
have $b\in aA$. Therefore, $aA=bA$ and Proposition~\ref{prop:aux} implies $a=bu$ for some invertible element $u$. Similarly, using right regular
representation $R:A \to \End_{F}\left(A\right)^{\textrm{op}}$,
where $\End_{F}\left(A\right)^{\textrm{op}}$ denotes the
opposite algebra of $\End_{F}\left(A\right)$, we obtain $Aa=Ab$. By a left-hand sided version of Proposition~\ref{prop:aux} we get $a=vb$ for some invertible element $v$.
This shows that $a \sim b$.

Now, suppose $A\cong\prod_{i=1}^{n}A_{i}$, where each $A_{i}$ is
an algebra over some field. Let $a=\left(a_{i}\right)_{i}$ and $b=\left(b_{i}\right)_{i}$
be elements of $A$, such that $a\approx_{A}b$. Applying \ref{enu:arrow-1}
to canonical projection $A \to A_{i}$, we get $a_{i}\approx_{A_{i}}b_{i}$,
hence $a_{i}\sim b_{i}$ by the first part of the proof. It is easy
to see that this implies $a\sim b$.
\end{proof}

Regarding Proposition~\ref{prop:best_possible} we pose the following open question.

\begin{question}\label{que:all}
Does the conclusion of Proposition~\ref{prop:best_possible} hold for any finite unital ring $A$?
\end{question}

To answer Question~\ref{que:all} it would suffice to consider the endomorphism ring $\End_\mathbb{Z}(H_p)$ of a finite abelian $p$-group $H_p$ (where $p$ is a prime), i.e. a group of the form
\[ H_p=\mathbb{Z}/{p^{k_1}}\mathbb{Z}\times\mathbb{Z}/{p^{k_2}}\mathbb{Z}\times\ldots\times\mathbb{Z}/{p^{k_m}}\mathbb{Z} \]
for some positive integers $k_1\geq k_2 \geq \ldots \geq k_m$. Indeed, if $A$ is an arbitrary finite unital ring, then by the classification of finite abelian groups, the additive group of $A$ is isomorphic to a direct product
\[ H_{p_1}\times H_{p_2}\times\ldots\times H_{p_n}, \]
where $p_1,p_2,\ldots,p_n$ are distinct prime numbers and each $H_{p_i}$ is a finite abelian $p_i$-group. Since $p_i$ are distinct, it follows that
\[ \End_\mathbb{Z}(A) \cong \End_\mathbb{Z}(H_{p_1}) \times \End_\mathbb{Z}(H_{p_2}) \times \ldots \times \End_\mathbb{Z}(H_{p_n}). \]
If every $\End_\mathbb{Z}(H_{p_i})$ satisfies the conclusion of Proposition~\ref{prop:best_possible}, then so does their direct product $\End_\mathbb{Z}(A)$ (see the end of the proof of Proposition~\ref{prop:best_possible}). Hence, a similar argument as in the proof of Proposition~\ref{prop:best_possible} would show that $A$ also satisfies the conclusion of Proposition~\ref{prop:best_possible}.

We remark that the structure of $\End_\mathbb{Z}(H_p)$ and the invertible elements of $\End_\mathbb{Z}(H_p)$ were described by Hillar and Rhea in \cite{Hil-Rhe}.

Propositions~\ref{prop:noncommutative--equivalence} and \ref{prop:best_possible} motivate us to define a compressed zero-divisor graph of finite unital rings in
the following way.

\begin{definition}\label{def:theta}
For a finite unital ring $R$, $\varTheta\left(R\right)$ is a directed
graph whose vertices are equivalence classes $\left[a\right]=\left[a\right]_{\sim}$ of elements of $R$, where $a \sim b$ if and only if $a=bu=vb$ for some units $u$ and $v$, and there is a directed edge $\left[x\right] \to \left[y\right]$ in $\varTheta(R)$ ($[x]$ and $[y]$ not necessarily distinct) if and only if $xy=0$.
We also define $\overline{\varTheta}\left(R\right)$ to be an undirected graph whose vertices are equivalence classes $\left[a\right]=\left[a\right]_{\sim}$ of elements of $R$
and there is an edge joining $\left[x\right]$ and $\left[y\right]$
(not necessarily distinct) if and only if either $xy=0$ or $yx=0$.
\end{definition}

Observe that edges in $\varTheta\left(R\right)$ and $\overline{\varTheta}\left(R\right)$ are well-defined
because if $x_{1}\sim y_{1}$, $x_{2}\sim y_{2}$, and $x_{1}x_{2}=0$,
then $y_{1}=vx_{1}$ and $y_{2}=x_{2}u$ for some units $u,v$,
so $y_{1}y_{2}=\left(vx_{1}\right)\left(x_{2}u\right)=v\left(x_{1}x_{2}\right)u=0$. However, unlike in the commutative setting, the equivalence classes do not naturally form a monoid, because the operation $[x]\cdot[y]=[xy]$ is not well-defined in general.

\begin{remark}\label{rem:NC-C}
Let $K$ be a commutative ring. If we ignore the directions of edges in the directed graph $\varTheta\left(K\right)$ and then replace every double edge by a single edge, we obtain the undirected graph $\overline{\varTheta}\left(K\right)$, which is actually equal to the undirected compressed zero-divisor graph of $K$ as defined in \cite{Dju-Jev-Sto}.
\end{remark}

\begin{proposition}
The mapping $R\mapsto\varTheta\left(R\right)$ extends to
a functor $\varTheta:\mathbf{FinRing}\to\mathbf{Digraph}$.
\end{proposition}

\begin{proof}
Let $f:R\to S$ be a unital ring homomorphism, where $R$
and $S$ are finite unital rings. Define $\varTheta\left(f\right):\varTheta\left(R\right)\to\varTheta\left(S\right)$
by $\varTheta\left(f\right)\left(\left[x\right]\right)=\left[f\left(x\right)\right]$.
Observe that $\varTheta\left(f\right)$ is well-defined since $f$
preserves units, and it is easy to verify that $\varTheta\left(f\right)$ is a graph
homomorphism. Clearly, $\varTheta\left(\textrm{id}_{R}\right)=\textrm{id}_{\varTheta\left(R\right)}$
and $\varTheta\left(g\circ f\right)=\varTheta\left(g\right)\circ\varTheta\left(f\right)$
for $f:R\to S$ and $g:S\to T$. So $\varTheta:\mathbf{FinRing}\to\mathbf{Digraph}$
is a functor.
\end{proof}

Observe that both categories involved have all finite products. Binary
product in the category $\mathbf{FinRing}$ is the direct product of
rings, while binary product in category $\mathbf{Digraph}$
is the tensor product of graphs (see \cite{Ham-Imr-Kla}). Recall that for graphs
$G$ and $H$, their tensor product $G\times H$ is defined as follows.
The set of vertices of $G\times H$ is the Cartesian product $V\left(G\right)\times V\left(H\right)$
and there is an edge $\left(g,h\right)\rightarrow\left(g',h'\right)$
if and only if there are edges $g\rightarrow g'$ and $h\rightarrow h'$
in $G$ and $H$ respectively.
Final object in the category $\mathbf{FinRing}$ is the
zero ring $0$ and final object in the category $\mathbf{Digraph}$
is the graph with precisely one vertex and one directed loop.

%Observe that both categories involved have all finite products. Binary
%product in the category $\mathbf{Ring}$ is the direct product of
%rings -- Cartesian product of underlying sets equipped with operations
%defined coordinate-wise. The same binary product is in $\mathbf{FinRing}$,
%since the direct product of two finite unital rings is again finite
%unital ring. Final object in the category $\mathbf{Ring}$ is the
%zero ring $0$. Binary product in the category $\mathbf{Digraph}$
%is the tensor product of graphs\underline{{[}Hammack{]}}. For graphs
%$G$ and $H$, their tensor product $G\times H$ is defined as follows.
%The set of vertices of $G\times H$ is Cartesian product $V\left(G\right)\times V\left(H\right)$
%and there is an edge $\left(g,h\right)\rightarrow\left(g',h'\right)$
%if and only if there are edges $g\rightarrow g'$ and $h\rightarrow h'$
%in $G$ and $H$ respectively. Final object in the category $\mathbf{Digraph}$
%is the graph with precisely one vertex and one directed loop.

\begin{proposition}\label{prop:noncommutative--preserves--products}
The functor $\varTheta:\mathbf{FinRing}\to\mathbf{Digraph}$
preserves finite products.
\end{proposition}

\begin{proof}
It is sufficient to show that functor $\varTheta$ preserves binary products and
final object. Let $R,S\in\mathrm{obj}\mathbf{FinRing}$. Since operations
in $R\times S$ are defined coordinate-wise, we have that $\left[\left(x,y\right)\right]=\left[\left(x',y'\right)\right]$
in $R\times S$ if and only if both $\left[x\right]=\left[x'\right]$
and $\left[y\right]=\left[y'\right]$. This shows that $V\left(\varTheta\left(R\times S\right)\right)=V\left(\varTheta\left(R\right)\right)\times V\left(\varTheta\left(S\right)\right)$.
There is an edge $\left[\left(x,y\right)\right]\rightarrow\left[\left(z,w\right)\right]$
in $\varTheta\left(R\times S\right)$ if and only if there are edges
$\left[x\right]\rightarrow\left[z\right]$ and $\left[y\right]\rightarrow\left[w\right]$
in $\varTheta\left(R\right)$ and $\varTheta\left(S\right)$ respectively.
Hence, $\varTheta\left(R\times S\right)$ is isomorphic to the tensor
product of graphs $\varTheta\left(R\right)$ and $\varTheta\left(S\right)$, where the isomorphism is given by $\left[\left(x,y\right)\right]\mapsto\left(\left[x\right],\left[y\right]\right)$.
Clearly, $\varTheta\left(0\right)$, the graph of the zero ring, is
the graph with precisely one vertex and one loop.
\end{proof}

Given an edge $a\rightarrow b$ in a directed graph $G$, we will
say that vertex $a$ is a \emph{pre-neighbour} of vertex $b$, and vertex $b$ a \emph{post-neighbour} of vertex $a$.
For a given set $X\subseteq V\left(G\right)$, let $N^{+}\left(X\right)$
denote the set of all common post-neighbours of all vertices in $X$.
Similarly, $N^{-}\left(X\right)$ will denote the set of all common
pre-neighbours of all vertices in $X$. For convenience we will simply
write $N^{+}\left(v\right)$ (respectively $N^{-}\left(v\right)$)
instead of $N^{+}\left(\left\{ v\right\} \right)$ (respectively $N^{-}\left(\left\{ v\right\} \right)$),
where $v\in V\left(G\right)$. We remark that $N^{-}\left(\left\{ v\right\} \right)$ and $N^{+}\left(\left\{ v\right\} \right)$ may or may not contain $v$ depending on whether there is a loop on $v$. Observe that in a graph $G$ with no equally oriented multiple edges we have $\left|N^{+}\left(v\right)\right|=\deg^{+}\left(v\right)$
and $\left|N^{-}\left(v\right)\right|=\deg^{-}\left(v\right)$, where $\deg^{+}\left(v\right)$ and $\deg^{-}\left(v\right)$ denote the outdegree and the indegree of $v$ respectively. All graphs considered in this paper have this property.

Next proposition generalizes \cite[Theorem~3.6]{Dju-Jev-Sto}. The proof is essentially the same except that extra care is needed due to the lack of commutativity.

\begin{proposition}\label{prop:noncommutative--preproduct}
Suppose $K,L_{1},L_{2}\in\mathrm{obj}\mathbf{FinRing}$
such that $\varTheta\left(K\right)\cong\varTheta\left(L_{1}\right)\times\varTheta\left(L_{2}\right)$.
Then $K=K_{1}\times K_{2}$ for some subrings $K_{1},K_{2}\subseteq K$
with $\varTheta\left(K_{1}\right)\cong\varTheta\left(L_{1}\right)$
and $\varTheta\left(K_{2}\right)\cong\varTheta\left(L_{2}\right)$.
\end{proposition}

\begin{proof}
If $\varTheta\left(L_{1}\right)\cong\varTheta\left(0\right)$,
then $\varTheta\left(L_{1}\right)\times\varTheta\left(L_{2}\right)\cong\varTheta\left(L_{2}\right)$
so we may take $K_{1}=0$ and $K_{2}=K$. We argue similarly if $\varTheta\left(L_{2}\right)\cong\varTheta\left(0\right)$.
So assume $\varTheta\left(L_{1}\right)\ncong\varTheta\left(0\right)$
and $\varTheta\left(L_{2}\right)\ncong\varTheta\left(0\right)$.

Let $f:\varTheta\left(L_{1}\right)\times\varTheta\left(L_{2}\right)\to\varTheta\left(K\right)$
be any isomorphism. Choose $k_{1},k_{2}\in K$ such that $f\left(\left(\left[1\right],\left[0\right]\right)\right)=\left[k_{1}\right]$
and $f\left(\left(\left[0\right],\left[1\right]\right)\right)=\left[k_{2}\right]$.
Observe that $N^{+}\left(\left(\left[1\right],\left[0\right]\right)\right)=N^{-}\left(\left(\left[1\right],\left[0\right]\right)\right)$,
hence $N^{+}\left(\left[k_{1}\right]\right)=N^{-}\left(\left[k_{1}\right]\right)$.
This implies that $\ann_{r}\left(k_{1}\right)=\ann_{\ell}\left(k_{1}\right)$.
Similarly, $\ann_{r}\left(k_{2}\right)=\ann_{\ell}\left(k_{2}\right)$.
Define 
\begin{equation}
K_{1}=\ann_{r}\left(k_{2}\right)=\ann_{\ell}\left(k_{2}\right) \quad\textup{and} \quad K_{2}=\ann_{r}\left(k_{1}\right)=\ann_{\ell}\left(k_{1}\right).\label{eq:24}
\end{equation}
Clearly, $K_{1}$ and $K_{2}$ are two-sided ideals of $K$. If $x\in K_{1}\cap K_{2}$,
then 
\begin{align*}
\left[x\right]\in N^{+}\left(\left[k_{1}\right]\right)\cap N^{+}\left(\left[k_{2}\right]\right) &=f\left(N^{+}\left(\left(\left[1\right],\left[0\right]\right)\right)\cap N^{+}\left(\left(\left[0\right],\left[1\right]\right)\right)\right)= \\
&=f\left(\left\{ \left(\left[0\right],\left[0\right]\right)\right\} \right)=\{[0]\}.
\end{align*}
Thus, $K_{1}\cap K_{2}=0$.

Observe that the subgraph of $\varTheta\left(L_{1}\right)\times\varTheta\left(L_{2}\right)$,
induced by $N^{+}\left(\left(\left[0\right],\left[1\right]\right)\right)$,
is isomorphic to $\varTheta\left(L_{1}\right)$. Hence, the
subgraph $G_{1}$ of $\varTheta\left(K\right)$, induced by $N^{+}\left(\left[k_{2}\right]\right)$,
is also isomorphic to $\varTheta\left(L_{1}\right)$. By definition of $K_1$ we clearly have
$V\left(G_{1}\right)=\left\{ \left[x\right]\in V\left(\varTheta\left(K\right)\right):x\in K_{1}\right\} $
and $k_{1}\in K_{1}$. Since $K_{1}$ is an ideal, it thus follows that $\left[k_{1}^{2}\right]\in V\left(G_{1}\right)$. Observe that $[k_1^2] \neq [0]$, since $[k_1]$ has no loop due to the fact that $L_1 \neq 0$.
Suppose $\left[k_{1}^{2}\right]\neq\left[k_{1}\right]$. Then $\left[k_{1}^{2}\right]\in V\left(G_{1}\right)\setminus\left\{ \left[0\right],\left[k_{1}\right]\right\} $.
Since $G_{1}\cong\varTheta\left(L_{1}\right)$, the neighbourhood
$N_{G_{1}}^{+}\left(V\left(G_{1}\right)\right)$ contains only one
vertex, i.e. $\left[0\right]$, and there is only one vertex in $G_{1}$
whose only post-neighbour in $G_{1}$ is $\left[0\right]$, i.e. $f\left(\left(\left[1\right],\left[0\right]\right)\right)=\left[k_{1}\right]$.
This implies that $\left[k_{1}^{2}\right]$ has a post-neighbour in
$G_{1}$ different from $\left[0\right]$, say $\left[a\right]$,
where $a\in K_{1}$. Hence, $k_{1}^{2}a=0$ because $G_{1}$ is an
induced subgraph of $\varTheta\left(K\right)$. This implies that $\left[k_{1}a\right]$
is a post-neighbour of $\left[k_{1}\right]$ in $\varTheta\left(K\right)$,
and since $K_{1}$ is an ideal, $\left[k_{1}a\right]\in V\left(G_{1}\right)$.
Therefore, $k_{1}a=0$ by the above. Similarly, this implies that
$\left[a\right]$ is a post-neighbour of $\left[k_{1}\right]$, hence
$a=0$, a contradiction. We have thus shown that $\left[k_{1}^{2}\right]=\left[k_{1}\right]$.
In particular, $k_{1}=k_{1}^{2}u_{1}$ for some unit $u_{1}\in K$.

Observe that $k_{1}\left(1-k_{1}u_{1}\right)=0$, hence $1-k_{1}u_{1}\in K_{2}$
by (\ref{eq:24}). If $1-k_{1}u_{1}=0$, then $k_{1}$ is a unit in
$K$, hence $\left[k_{1}\right]=\left[1\right]$. But this would imply
that $\left[0\right]$ is the only post-neighbour of $\left[k_{1}\right]$
in $\varTheta\left(K\right)$, which would further imply
$K_{2}=0$. In this case, $\varTheta\left(L_{2}\right)\cong\varTheta\left(0\right)$,
a contradiction. So $1-k_{1}u_{1}\neq0$.

Suppose $\left[1-k_{1}u_{1}\right]\neq\left[k_{2}\right]$. Let $G_{2}$
be the subgraph of $\varTheta\left(K\right)$, induced by $N^{+}\left(\left[k_{1}\right]\right)$.
Then the same argument as above shows that $\left[1-k_{1}u_{1}\right]\in V\left(G_{2}\right)\setminus\left\{ \left[0\right],\left[k_{2}\right]\right\} $
has a post-neighbour in $G_{2}$ different from $\left[0\right]$,
say $\left[b\right]$, where $0 \neq b\in K_{2}$. Hence, 
\begin{equation}
\left(1-k_{1}u_{1}\right)b=0\label{eq:24-1}
\end{equation}
 because $G_{2}$ is an induced subgraph of $\varTheta\left(K\right)$.
Since $k_{2}k_{1}=0$, we have $k_{2}=k_{2}\left(1-k_{1}u_{1}\right)$.
Hence, $k_{2}b=0$ by (\ref{eq:24-1}). This implies $b\in K_{1}$,
so $b\in K_{1}\cap K_{2}=0$, a contradiction. Thus, $\left[1-k_{1}u_{1}\right]=\left[k_{2}\right]$,
and consequently $1=k_{1}u_{1}+k_{2}u_{2}$ for some unit $u_{2}\in K$. This
shows that $K=K_{1}+K_{2}$. Since we already know that $K_{1}\cap K_{2}=0$,
we conclude that $K=K_{1}\times K_{2}$.

Observe that if $x\in K_{1}$ and $x\sim y$ in $K$, then $y\in K_{1}$
and $x\sim y$ in $K_{1}$. Hence, $\varTheta\left(K_{1}\right)\cong G_{1}\cong\varTheta\left(L_{1}\right)$
and similarly $\varTheta\left(K_{2}\right)\cong\varTheta\left(L_{2}\right)$.
\end{proof}

It follows from Remark~\ref{rem:NC-C} and \cite[Example~3.5]{Dju-Jev-Sto} that functor $\varTheta$ does not preserve limits.

\section{Rings of matrices over finite fields}\label{sec:matrices}

We now investigate the graph of the matrix ring $M_{n}\left(F\right)$, where $F$ is a finite field. In particular, we determine the number of vertices, their degrees, existence of Hamiltonian paths and cycles, directed cliques of maximal size, and dominating sets of minimal size.

Below we first establish a bijective correspondence
between the set of vertices of $\varTheta\left(M_{n}\left(F\right)\right)$ and the set of ordered
pairs $\left(V,W\right)$ of subspaces of $F^{n}$ such that $\dim V+\dim W=n$. This correspondence is given by the map $[A] \mapsto (\textup{im} A, \textup{ker} A)$, where $\textup{im} A$ and $\textup{ker} A$ denote respectively the image and the kernel of the linear transformation $x \mapsto Ax$.

Assume $A\sim B$ in $M_{n}\left(F\right)$. Then $A=BP=QB$ for some invertible matrices $P$ and $Q$. The invertibility of $P$ clearly implies that $\im A=\im B$ and the invertibility of $Q$ implies that $\ker A=\ker B$.
So for $\left[A\right]\in V\left(\varTheta\left(M_{n}\left(F\right)\right)\right)$,
there is a well-defined ordered pair $\left(\im A,\ker A\right)$
of subspaces of $F^n$ whose dimensions add up to $n$.
%If $y=Ax$
%where $x,y\in F^{n}$, then $y=Ax=\left(BP\right)x=B\left(Px\right)$,
%and similarly, if $z=Bx$ then $z=A\left(P^{-1}x\right)$. Hence $\imA=\imB$. Also if $Bx=0$, where $x \in F^{n}$,
%then $Ax=\left(QB\right)x=Q\left(Bx\right)=0$ and similarly if $Ax=0$
%then $Bx=0$ hence $\kerA=\kerB$.

Conversely, given a pair $\left(V,W\right)$ of subspaces of $F^{n}$
such that $\dim V+\dim W=n$, fix some bases of $V$ and $W$,
say $\mathcal{B}_{V}$ and $\mathcal{B}_{W}$. Let $\mathcal{B}$ be some basis of $F^{n}$ that
contains $\mathcal{B}_{W}$. Let $A$ be a matrix that represents (in the standard basis) a linear transformation which maps $\mathcal{B}\setminus \mathcal{B}_{W}$ bijectively onto $\mathcal{B}_{V}$ and maps $\mathcal{B}_{W}$
to $0$. Clearly, $\left(\im A,\ker A\right)=\left(V,W\right)$.
If $B$ is another matrix with $\left(\im B,\ker B\right)=\left(V,W\right)$,
then the equality $\im A=\im B$ implies $A=BP$ for
some invertible matrix $P$ and the equality $\ker A=\ker B$
implies $A=QB$ for some invertible matrix $Q$. This establishes the aforementioned bijective correspondence.

We will use the above correspondence throughout the paper. Observe also that, by the this correspondence, there is an edge $\left(V_{1},W_{1}\right)\rightarrow\left(V_{2},W_{2}\right)$
if and only if $W_{1}\supseteq V_{2}$.

Let $q>1$ be a prime power. We will denote the number of $k$-dimensional subspaces of $F_{q}^{n}$
by $\binom{n}{k}_{q}$. This is usually called a \emph{$q$-binomial
coefficient}. In particular, $\binom{n}{k}_{q}=0$ if $k>n$ or $k<0$.
For a detailed treatment of $q$-binomial coefficients we refer the reader to Stanley's book \cite[\S 7.1]{Sta}, where, in particular, it is shown that the $q$-binomial coefficients satisfy the following equalities
\begin{equation}\label{eq:razpisano}
 \binom{n}{k}_{q}=\frac{\left(q^{n}-1\right)\left(q^{n}-q\right)\cdots\left(q^{n}-q^{k-1}\right)}{\left(q^{k}-1\right)\left(q^{k}-q\right)\cdots\left(q^{k}-q^{k-1}\right)},
\end{equation}
\begin{equation}\label{eq:simetrija}
\binom{n}{k}_{q}=\binom{n}{n-k}_{q},
\end{equation}
\begin{equation}\label{eq:rekurzija}
\binom{n}{k}_{q}=\binom{n-1}{k}_{q}+q^{n-k}\binom{n-1}{k-1}_{q}.
\end{equation}
We remark that Stanley's book also gives a generating function for $\binom{n}{k}_{q}$.

We begin by counting the vertices in $\varTheta\left(M_{n}\left(F_{q}\right)\right)$ and determining their degrees.

\begin{proposition}\label{prop:degree}
If $R=M_{n}\left(F_{q}\right)$, then 
\[
\left|V\left(\varTheta\left(R\right)\right)\right|=\sum_{i=0}^{n}\binom{n}{i}_{q}^{2}
\]
and 
\[
\deg^{+}\left(\left(V,W\right)\right)=\deg^{-}\left(\left(V,W\right)\right)=\sum_{i=0}^{\dim W}\binom{\dim W}{i}_{q}\binom{n}{i}_{q}.
\]
\end{proposition}

\begin{proof}
The number of pairs $(V,W)$ of subspaces of $F_q^n$, where $\dim V+\dim W=n$ and $\dim V=i$ equals $\binom{n}{i}_q\binom{n}{n-i}_q=\binom{n}{i}_q^2$. Summing over all $i$ gives the formula for the number of vertices.

A vertex $(V',W')$ is a post-neighbour of $(V,W)$ if and only if $V' \subseteq W$. There are $\binom{\dim W}{i}_q$ subspaces of $W$ with dimension equal to $i$ and for each such subspace $X$ there are $\binom{n}{n-i}_q=\binom{n}{i}_q$ vertices of the form $(X,*)$.
Hence, $(V,W)$ has $\binom{\dim W}{i}_q\binom{n}{i}_q$ post-neighbours $(V',W')$ with $\dim V'=i$. Summing over all $0\leq i\leq \dim W$ gives
\[ \deg^{+}\left(\left(V,W\right)\right)=\sum_{i=0}^{\dim W}\binom{\dim W}{i}_{q}\binom{n}{i}_{q}. \]

A vertex $(V',W')$ is a pre-neighbour of $(V,W)$ if and only if $W' \supseteq V$. The number of spaces of $F_q^n$ with dimension $j$ that contain $V$,
is equal to the number of subspaces of $F_q^n/V \cong F_q^{n-\dim V}$ with dimension $j-\dim V$, and this equals $\binom{n-\dim V}{j-\dim V}_q=\binom{n-\dim V}{n-j}_q$.
For each subspace $Y$ of the former kind there are $\binom{n}{n-j}_q$ vertices of the form $(*,Y)$. Hence, $(V,W)$ has $\binom{n-\dim V}{n-j}_q\binom{n}{n-j}_q$ pre-neighbours $(V',W')$ with $\dim W'=j$.
Summing over all $\dim V\leq j \leq n$ gives
\[ \deg^{-}\left(\left(V,W\right)\right)=\sum_{j=\dim V}^{n}\binom{n-\dim V}{n-j}_q\binom{n}{n-j}_q. \]
If we substitute $j=n-i$ into this sum we obtain
\[ \deg^{-}\left(\left(V,W\right)\right)=\sum_{i=0}^{n-\dim V}\binom{n-\dim V}{i}_q\binom{n}{i}_q=\sum_{i=0}^{\dim W}\binom{\dim W}{i}_q\binom{n}{i}_q. \]
\end{proof}

This already enables us to show that non-isomorphic matrix rings have non-isomorphic graphs.

\begin{proposition}\label{prop:same--graphs--same--rings}
If $\varTheta\left(M_{n}\left(F\right)\right)\cong\varTheta\left(M_{m}\left(E\right)\right)$,
where $n>1$ and $F$ and $E$ are finite fields, then $n=m$ and $F\cong E$.
\end{proposition}

\begin{proof}
Let $G=\varTheta\left(M_{n}\left(F\right)\right)$ and $F\cong F_{q}$, $E\cong F_{q'}$. Observe that, 
by Proposition~\ref{prop:degree}, $\deg^{+}\left(\left(V,W\right)\right)$ only depends on $\dim W$, and as a function of $\dim W$ it is strictly increasing. 
%Observe
%that, by Proposition~\ref{prop:degree}, we have $\deg^{+}\left(\left(V,W\right)\right)=\deg^{+}\left(\left(X,Y\right)\right)$
%if and only if $\dim W=\dim Y$. Clearly, if $\dim W=\dim Y$
%then $\deg^{+}\left(\left(V,W\right)\right)=\deg^{+}\left(\left(X,Y\right)\right)$.
%Conversely, if $\dim W<\dim Y$ then
%\begin{eqnarray*}
%\deg^{+}\left(\left(V,W\right)\right) & = & \sum_{i=0}^{\dim W}\left(\begin{array}{c}
%\dim W\\
%i
%\end{array}\right)_{q}\left(\begin{array}{c}
%n\\
%i
%\end{array}\right)_{q}<\sum_{i=0}^{\dim W}\left(\begin{array}{c}
%\dim Y\\
%i
%\end{array}\right)_{q}\left(\begin{array}{c}
%n\\
%i
%\end{array}\right)_{q}<\\
% & < & \sum_{i=0}^{\dim Y}\left(\begin{array}{c}
%\dim Y\\
%i
%\end{array}\right)_{q}\left(\begin{array}{c}
%n\\
%i
%\end{array}\right)_{q}=\deg^{+}\left(\left(X,Y\right)\right).
%\end{eqnarray*}
Hence, the number of different outdegrees in $G$ is equal to the number of different
dimensions of subspaces of $F^{n}$ which is $n+1$. This implies $n=m$.
In addition, the least outdegree in $G$ different from $1$ is equal
to
\begin{align*}
d &=\sum_{i=0}^{1}\binom{1}{i}_{q}\binom{n}{i}_{q} = 1+\binom{n}{1}_{q}=1+\frac{q^{n}-1}{q-1}=\\
&= q^{n-1}+q^{n-2}+\cdots+q^{2}+q+2.
\end{align*}
Every term $q^{i}$, $i\geq 1$, is an increasing function of $q\in\mathbb{N}$. Given that $n>1$, this implies that $d$ is an increasing function of $q$, so it is injective.
Since $n=m$, this implies $q=q'$, hence $F\cong E$.
\end{proof}

In the next few results we discuss Hamiltonian cycles and paths in graph $\varTheta(M_n(F))$. Recall that a cycle in a directed graph is called \emph{simple} if no vertex in this cycle is repeated. A cycle of length $1$ is a vertex with a loop.

\begin{lemma}\label{lem:cycle--subspace}
Let $V$ be a nontrivial proper subspace
of $F^{n}$ and $\mathscr{S}$ a set of subspaces of $F^{n}$, such
that $V\notin\mathscr{S}$. Then, there is a simple (possibly empty)
cycle in $\varTheta\left(M_{n}\left(F\right)\right)\setminus\left\{ \left[0\right],\left[1\right]\right\} $
that contains all vertices of the form $\left(V,W\right)$ or $\left(W,V\right)$
where $W\notin\mathscr{S}$, and every edge in this cycle is of the
form $\left(X_{1},X_{2}\right)\rightarrow\left(X_{2},X_{3}\right)$.
\end{lemma}

\begin{proof}
Let $\mathscr{W}=\left\{ W:W\textrm{ is subspace of }F^{n},\ \dim W+\dim V=n,\ W\notin\mathscr{S}\right\}$.
Since $\mathscr{W}$ is a finite set, we can enumerate all of its elements,
say $\mathscr{W}=\left\{ W_{1},W_{2},\ldots,W_{k}\right\} $. Then
the cycle 
\begin{align*}
\left(V,W_{1}\right) & \rightarrow  \left(W_{1},V\right)\rightarrow\left(V,W_{2}\right)\rightarrow\left(W_{2},V\right)\rightarrow\cdots\\
\cdots & \rightarrow  \left(V,W_{k}\right)\rightarrow\left(W_{k},V\right)\rightarrow\left(V,W_{1}\right)
\end{align*}
has all the desired properties except that it might not be simple.
Observe that it is simple unless $n$ is even and $\dim V=\frac{n}{2}$,
in which case $V$ is an element of $\mathfrak{\mathscr{W}}$, say $V=W_{1}$. In this
case, we replace $\left(V,W_{1}\right)\rightarrow\left(W_{1},V\right)\rightarrow\left(V,W_{2}\right)$
by $\left(V,V\right)\rightarrow\left(V,W_{2}\right)$ to make the
cycle simple.
\end{proof}

In what follows, we denote 
\[
S_{k,m}=\left\{ \left(V,W\right):V,W\textrm{ are subspaces of }F^{n},\ \dim V=k,\ \dim W=m\right\}.
\]

\begin{lemma}\label{lem:cycle--partition}
Let $n=k+m$ be a partition of $n$,
where $k,m\neq0$. Then there is a simple cycle in $\varTheta\left(M_{n}\left(F\right)\right)\setminus\left\{ \left[0\right],\left[1\right]\right\} $
that contains all the vertices in $S_{k,m}\cup S_{m,k}$, and every
edge in this cycle is of the form $\left(X_{1},X_{2}\right)\rightarrow\left(X_{2},X_{3}\right)$.
\end{lemma}

\begin{proof}
Let $\mathscr{V}=\left\{ V_{1},V_{2},\ldots,V_{s}\right\} $ be the
set of all subspaces of $F^{n}$ whose dimension is $k$.

Suppose that $k\neq m$. By Lemma~\ref{lem:cycle--subspace}, for
every $i\in\left\{ 1,2,\ldots,s\right\} $ there is a simple cycle
$\mathcal{C}_{i}$ in $\varTheta\left(M_{n}\left(F\right)\right)\setminus\left\{ \left[0\right],\left[1\right]\right\} $
containing all the vertices of the form $\left(V_{i},W\right)$ or
$\left(W,V_{i}\right)$, where $W$ is an arbitrary subspace of dimension $m$.
Fix a subspace $X$ whose dimension is $m$. Then $X\neq V_{i}$ since $m\neq k$,
so there is an edge in $\mathcal{C}_{i}$ of the form $\left(V_{i},X\right)\rightarrow\left(X,V_{i}\right)$.
Remove this edge from $\mathcal{C}_{i}$ to obtain a simple path.
We denote this path by $\left(X,V_{i}\right)\rightarrow\mathscr{\mathcal{P}}_{i}\rightarrow\left(V_{i},X\right)$.
Observe that these paths for different $i$ are disjoint since $k\neq m$, and their union
contains all the vertices in $S_{k,m}\cup S_{m,k}$. In this case,
the cycle 
\begin{align*}
\left(X,V_{1}\right) & \rightarrow  \mathscr{\mathcal{P}}_{1}\rightarrow\left(V_{1},X\right)\rightarrow\left(X,V_{2}\right)\rightarrow\mathscr{\mathcal{P}}_{2}\rightarrow\left(V_{2},X\right)\rightarrow\cdots\\
\cdots & \rightarrow  \left(X,V_{s}\right)\rightarrow\mathscr{\mathcal{P}}_{s}\rightarrow\left(V_{s},X\right)\rightarrow\left(X,V_{1}\right)
\end{align*}
has the desired property.

Now suppose $k=m$. For every $i\in\left\{ 1,2,\ldots,s\right\} $, by Lemma~\ref{lem:cycle--subspace} (with $\mathscr{S}=\left\{ V_{1},V_{2},\ldots,V_{i-1}\right\} $),
there is a simple cycle $\mathcal{D}_{i}$ in $\varTheta\left(M_{n}\left(F\right)\right)\setminus\left\{ \left[0\right],\left[1\right]\right\} $
that contains all the vertices of the form $\left(V_{i},W\right)$
or $\left(W,V_{i}\right)$, where $W\notin\left\{ V_{1},V_{2},\ldots,V_{i-1}\right\} $.
The cycle $\mathcal{D}_{i}$ contains the edge $\left(V_{i},V_{s}\right)\rightarrow\left(V_{s},V_{i}\right)$.
Remove this edge from $\mathcal{D}_{i}$ to obtain a simple path and denote
this path by $\left(V_{s},V_{i}\right)\rightarrow\mathcal{R}_{i}\rightarrow\left(V_{i},V_{s}\right)$
(when $i=s$, this path is just one vertex $\left(V_{s},V_{s}\right)$).
Observe that these paths for different $i$ are disjoint and their union contains all
the vertices in $S_{k,k}$. In this case, the cycle 
\begin{align*}
\left(V_{s},V_{1}\right) & \rightarrow  \mathcal{R}_{1}\rightarrow\left(V_{1},V_{s}\right)\rightarrow\left(V_{s},V_{2}\right)\rightarrow\mathcal{R}_{2}\rightarrow\left(V_{2},V_{s}\right)\rightarrow\cdots\\
\cdots & \rightarrow \left(V_{s},V_{s-1}\right)\rightarrow\mathcal{R}_{s-1}\rightarrow\left(V_{s-1},V_{s}\right)\rightarrow\left(V_{s},V_{s}\right)\rightarrow\left(V_{s},V_{1}\right)
\end{align*}
has the desired property.
\end{proof}

We are now ready to describe Hamiltonian paths and cycles in the zero-divisor graph of a matrix ring.

\begin{theorem} Let $F$ be a finite field.
\begin{enumerate}
\item \label{enu:hamilton--cycle}
If $n\in\left\{ 2,3\right\} $ then $\varTheta\left(M_{n}\left(F\right)\right)\setminus\left\{ \left[0\right],\left[1\right]\right\}$
contains a Hamiltonian cycle.
\item If $n\geq4$ then $\varTheta\left(M_{n}\left(F\right)\right)\setminus\left\{ \left[0\right],\left[1\right]\right\} $
contains a Hamiltonian path but $\overline{\varTheta}\left(M_{n}\left(F\right)\right)\setminus\left\{ \left[0\right],\left[1\right]\right\} $
does not contain a Hamiltonian cycle.
\end{enumerate}
\end{theorem}

\begin{proof}
If $\left(V,W\right)$ is a vertex in $\varTheta\left(M_{n}\left(F\right)\right)\setminus\left\{ \left[0\right],\left[1\right]\right\} $,
then $n=\dim V+\dim W$ is a partition of $n$ with
nonzero parts.
There is only one partition of $2$ with nonzero parts, namely $2=1+1$,
and only one partition of $3$ with nonzero parts, namely $3=1+2$.
Hence, \ref{enu:hamilton--cycle} follows directly from Lemma~\ref{lem:cycle--partition}.

Let $n\geq4$ and denote $G=\varTheta\left(M_{n}\left(F\right)\right)\setminus\left\{ \left[0\right],\left[1\right]\right\} $ and
$\overline{G}=\overline{\varTheta}\left(M_{n}\left(F\right)\right)\setminus\left\{ \left[0\right],\left[1\right]\right\} $.
Observe that in $\overline{G}\setminus S_{1,n-1}$ there is no edge
incident with any vertex in $S_{n-1,1}$. Hence, every element of $S_{n-1,1}$
is an isolated vertex of $\overline{G}\setminus S_{1,n-1}$ and since
$n\geq4$ there is at least one vertex in $\overline{G}\setminus\left(S_{1,n-1}\cup S_{n-1,1}\right)$.
Hence, there are at least $\left|S_{n-1,1}\right|+1=\left|S_{1,n-1}\right|+1$
connected components in $\overline{G}\setminus S_{1,n-1}$, therefore
$\overline{G}$ does not contain a Hamiltonian cycle by \cite[Theorem~4.2]{Bon-Mur}).

By Lemma~\ref{lem:cycle--partition}, there is a simple cycle $\mathcal{E}_{k}$
in $G$ containing all vertices in $S_{k,m}\cup S_{m,k}$, where $k+m=n$.
Observe that cycles $\mathcal{E}_{1},\mathcal{E}_{2},\ldots,\mathcal{E}_{\left\lfloor \frac{n}{2}\right\rfloor }$
are disjoint and their union contains all the vertices of $G$. Choose
a chain of subspaces $Y_{1}\subseteq Y_{2}\subseteq\cdots\subseteq Y_{\left\lfloor \frac{n}{2}\right\rfloor }$,
where $\dim Y_{k}=k$. Then there is an edge in $\mathcal{E}_{k}$
of the form $\left(X_{k},Y_{k}\right)\rightarrow\left(Y_{k},Z_{k}\right)$
for some $X_{k},Z_{k}$. Remove this edge from $\mathcal{E}_{k}$
to obtain a path, which we denote by $\left(Y_{k},Z_{k}\right)\rightarrow\mathcal{Q}_{k}\rightarrow\left(X_{k},Y_{k}\right)$.
Then the path 
\begin{align*}
\left(Y_{\left\lfloor \frac{n}{2}\right\rfloor },Z_{\left\lfloor \frac{n}{2}\right\rfloor }\right) & \rightarrow  \mathcal{Q}_{\left\lfloor \frac{n}{2}\right\rfloor }\rightarrow\left(X_{\left\lfloor \frac{n}{2}\right\rfloor },Y_{\left\lfloor \frac{n}{2}\right\rfloor }\right)\rightarrow\\
\rightarrow\left(Y_{\left\lfloor \frac{n}{2}\right\rfloor -1},Z_{\left\lfloor \frac{n}{2}\right\rfloor -1}\right) & \rightarrow  \mathcal{Q}_{\left\lfloor \frac{n}{2}\right\rfloor -1}\rightarrow\left(X_{\left\lfloor \frac{n}{2}\right\rfloor -1},Y_{\left\lfloor \frac{n}{2}\right\rfloor -1}\right)\rightarrow\\
&\rightarrow\cdots\rightarrow\\
\rightarrow  \left(Y_{2},Z_{2}\right) & \rightarrow\mathcal{Q}_{2}\rightarrow\left(X_{2},Y_{2}\right)\rightarrow\\
\rightarrow  \left(Y_{1},Z_{1}\right) & \rightarrow\mathcal{Q}_{1}\rightarrow\left(X_{1},Y_{1}\right)
\end{align*}
is a Hamiltonian path in $G$.
\end{proof}

Next we determine the size of the largest directed clique in $\varTheta\left(M_{n}\left(F\right)\right)$. By a \emph{directed clique}, we mean a subgraph in which for any two (not necessarily
distinct) vertices $u$ and $v$, there is a directed edge from $u$ to $v$.
For a subspace $U\subseteq F^{n}$, let 
\[ K\left(U\right) =\{\left(V,W\right) \in V(\varTheta\left(M_{n}\left(F\right)\right)) : V\subseteq U\subseteq W\}. \]
It is easily checked that $K\left(U\right)$
is a directed clique in $\varTheta\left(M_{n}\left(F\right)\right)$.

\begin{lemma}\label{lem:max--clique}
Every directed clique in $\varTheta\left(M_{n}\left(F\right)\right)$
is contained in $K\left(U\right)$ for some subspace $U\subseteq F^{n}$.
\end{lemma}

\begin{proof}
Let $K=\left\{ \left(V_{i},W_{i}\right):1\leq i\leq k \right\} $ be any directed clique in $\varTheta\left(M_{n}\left(F\right)\right)$. Take $U=\bigcap_{i=1}^{k}W_{i}$. Since
$K$ is a directed clique, we have $V_{j}\subseteq W_{i}$ for all
$i,j\in\left\{ 1,2,\ldots,k\right\} $. This implies $V_{j}\subseteq U \subseteq W_{j}$ for every $j\in\left\{ 1,2,\ldots,k\right\}$. We conclude that $K \subseteq K(U)$.
\end{proof}

\begin{proposition}
The size of the largest directed clique in $\varTheta\left(M_{n}\left(F_q\right)\right)$ is 
\[ \sum_{i=0}^{\left\lfloor \frac{n}{2}\right\rfloor}\binom{\left\lfloor \frac{n}{2}\right\rfloor}{i}_{q}\binom{\left\lceil \frac{n}{2}\right\rceil}{i}_{q} \]
and any directed clique of this size is of the form $K(U)$, where $\dim U$ is equal to $\left\lfloor \frac{n}{2}\right\rfloor $ or $\left\lceil \frac{n}{2}\right\rceil$.
\end{proposition}

\begin{proof}
By Lemma~\ref{lem:max--clique}, it suffices to determine the maximum of $|K(U)|$, where $U$ is a subspace of $F_{q}^{n}$.
Let $u=\dim U$. Subspaces containing $U$ are in a bijective correspondence with the subspaces of $F_q^n/U$, hence
\[
\left|K\left(U\right)\right|=\sum_{i=0}^{\min\{u,n-u\}}\binom{u}{i}_{q}\binom{n-u}{i}_{q}.
\]
Due to the symmetry in the formula for $|K(U)|$ we may assume $u \leq \left\lfloor \frac{n}{2}\right\rfloor$. In this case $\min\{u,n-u\}=u \leq \left\lfloor \frac{n}{2}\right\rfloor$.
Since $q$-binomial coefficients are nonnegative, it thus suffices to show that
\begin{equation}
\binom{u}{i}_{q}\binom{n-u}{i}_{q}\leq\binom{\left\lfloor \frac{n}{2}\right\rfloor }{i}_{q}\binom{n-\left\lfloor \frac{n}{2}\right\rfloor }{i}_{q}\label{eq:1-2}
\end{equation}
for all $i\in\left\{ 0,1,\ldots,u\right\}$.
Writing out $q$-binomial coefficients as in \eqref{eq:razpisano} and multiplying the inequality
by $\prod_{r=0}^{i-1}\left(q^{i}-q^{r}\right)^{2}$, we get
\[
\prod_{r=0}^{i-1}\left(q^{u}-q^{r}\right)\left(q^{n-u}-q^{r}\right)\leq\prod_{r=0}^{i-1}\left(q^{\left\lfloor \frac{n}{2}\right\rfloor }-q^{r}\right)\left(q^{n-\left\lfloor \frac{n}{2}\right\rfloor }-q^{r}\right).
\]
So it suffices to show that
\[
\left(q^{u}-q^{r}\right)\left(q^{n-u}-q^{r}\right)\leq\left(q^{\left\lfloor \frac{n}{2}\right\rfloor }-q^{r}\right)\left(q^{n-\left\lfloor \frac{n}{2}\right\rfloor }-q^{r}\right)
\]
for all $r\in\left\{ 0,1,\ldots,u-1\right\}$.
Multiplying out both sides of the last inequality and simplifying, we get
\[
q^{n-u+r}+q^{u+r}\geq q^{n-\left\lfloor \frac{n}{2}\right\rfloor+r}+q^{\left\lfloor \frac{n}{2}\right\rfloor +r}.
\]
We thus need to determine the minimum on the set of integers of the function $f \colon \mathbb{\mathbb{R}\to R}$ defined by $f\left(x\right)=q^{n-x+r}+q^{x+r}$.
Observe that $f'\left(x\right)=\left(q^{x+r}-q^{n-x+r}\right)\ln q$. From this it is easy to see that $f$ has a local minimum at $x=\frac{n}{2}$. In addition, $f'\left(x\right)<0$
for $x<\frac{n}{2}$ and $f'\left(x\right)>0$ for $x>\frac{n}{2}$. Hence, if we
restrict the domain of $f$ to integers, then its global minimum is equal to $f\left(\left\lfloor \frac{n}{2} \right\rfloor \right)=f\left(\left\lceil \frac{n}{2}\right\rceil \right)$ as required.
\end{proof}

The following question remains open.

\begin{question}
What is the size of the largest clique in $\overline{\varTheta}(M_n(F_q))$?
\end{question}

Recall that a \emph{dominating set} in an undirected graph $G$ is a subset $D \subseteq V(G)$ such that every vertex of $G$ is either in $D$ or is adjacent to at least one vertex in $D$.
The least possible size of a dominating set is called the \emph{domination number} of $G$. In next two propositions we determine the domination number of graph $\overline{\varTheta}(M_n(F_q))$ as well as what could be called the ``directed domination number'' of $\varTheta(M_n(F_q))$.

\begin{proposition}\label{prop:dominating}
Let $n\geq3$ and $R=M_{n}\left(F_{q}\right)$.
There is a subset $D\subseteq V\left(\varTheta\left(R\right)\right)$
with $\binom{n}{1}_{q}$ elements such that:
\begin{enumerate}
\item\label{enu:edge_dominating} $D$ is a dominating set for $\overline{\varTheta}\left(R\right)\setminus\left\{ \left[0\right],\left[1\right]\right\}$ of the least possible size,
\item\label{enu:edge--v--D} for every vertex $v\in V\left(\varTheta\left(R\right)\right)\setminus D$
there is an edge in $\varTheta\left(R\right)$ with source $v$ and
target in $D$,
\item\label{enu:edge--D--v} for every vertex $v\in V\left(\varTheta\left(R\right)\right)\setminus D$
there is an edge in $\varTheta\left(R\right)$ with target $v$ and
source in $D$.
\end{enumerate}
\end{proposition}

\begin{proof}
Let $\mathscr{V}=\left\{ V_{1},V_{2},\ldots,V_{k}\right\} $ be the
set of all $1$-dimensional subspaces of $F_{q}^{n}$ and let $\mathscr{W}=\left\{ W_{1},W_{2},\ldots,W_{k}\right\} $
be the set of all $\left(n-1\right)$-dimensional subspaces of $F_{q}^{n}$.
These sets indeed have the same size $k=\binom{n}{1}_{q}$.

Let $D$ be any dominating set for $\overline{\varTheta}\left(R\right)\setminus\left\{ \left[0\right],\left[1\right]\right\} $.
We will say that a subspace $X\subseteq F_{q}^{n}$ is an image in
$D$ (respectively a kernel in $D$) if $\left(X,Y\right)\in D$
(respectively $\left(Y,X\right)\in D$) for some subspace $Y\subseteq F_{q}^{n}$.

Fix some $i\in\left\{ 1,2,\ldots,k\right\} $. Suppose first that for some $i \in \{1,2,\ldots,k\}$, $W_{i}$ is
not a kernel in $D$. We claim that, in this case, for every $j\in\left\{ 1,2,\ldots,k\right\} $,
$V_{j}$ is either an image or a kernel in $D$. To see this, suppose $W_{i}$ is
not a kernel in $D$. If $\left(W_{i},V_{j}\right)\in D$
then $V_{j}$ is a kernel in $D$, otherwise $\left(W_{i},V_{j}\right)$
is adjacent to some vertex $\left(X,Y\right)\in D$. There is no edge
$\left(X,Y\right)\rightarrow\left(W_{i},V_{j}\right)$ since this
would imply $W_{i}\subseteq Y$ and hence $W_{i}=Y$ due to dimension
of $W_{i}$ being $n-1$. This would contradict the assumption that
$W_{i}$ is not a kernel in $D$. Hence, there is an edge $\left(W_{i},V_{j}\right)\rightarrow\left(X,Y\right)$
which implies $X\subseteq V_{j}$ and thus $X=V_{j}$ due to the dimension
of $V_{j}$ being $1$. This means that $V_{j}$ is an image in $D$, which proves our claim.
We conclude that in this case $\left|D\right|\geq k=\binom{n}{1}_{q}$, because vertices $\left(W_{s},V_{t}\right),\left(V_{t},W_{s}\right)$
for $s,t\in\left\{ 1,2,\ldots,k\right\} $ are all distinct due to our assumption
$n\geq3$.
On the other hand, if $W_{i}$ is a kernel in $D$ for all $i\in\left\{ 1,2,\ldots,k\right\} $,
then clearly $\left|D\right|\geq k=\binom{n}{1}_{q}$.

To finish the proof, we construct a set $D$ with $\binom{n}{1}_{q}$ vertices that satisfies \ref{enu:edge--v--D} and \ref{enu:edge--D--v}, and hence also \ref{enu:edge_dominating}.
Let $D=\left\{ \left(V_{i},W_{i}\right):i\in\left\{ 1,2,\ldots,k\right\} \right\} $.
If $\left(X,Y\right)$ is a vertex distinct from $\left[0\right]$
and $\left[1\right]$, then $X$ is a proper subspace of $F_{q}^{n}$
and $Y$ is a nonzero subspace of $F_{q}^{n}$. Hence, $X\subseteq W_{i}$ for some
$i$ and $V_{j}\subseteq Y$ for some $j$, so we have a path $\left(V_{i},W_{i}\right)\rightarrow\left(X,Y\right)\rightarrow\left(V_{j},W_{j}\right)$.
\end{proof}

\begin{proposition}
Let $n=2$ and $R=M_{n}\left(F_{q}\right)$. There is a subset $D\subseteq V\left(\varTheta\left(R\right)\right)$,
with $\binom{n}{1}_{q}-1$ elements and a vertex $d\in V\left(\varTheta\left(R\right)\right)\setminus D$
such that:
\begin{enumerate}
\item $D$ is a dominating set for $\overline{\varTheta}\left(R\right)\setminus\left\{ \left[0\right],\left[1\right]\right\}$ of the least possible size,
\item\label{enu:edge--v--D'} $D'=D\cup\left\{ d\right\} $ is a subset
of $V\left(\varTheta\left(R\right)\right)$ of the least possible size
 such that for every vertex $v\in V\left(\varTheta\left(R\right)\right)\setminus D'$
there is an edge in $\varTheta\left(R\right)$ with source $v$ and
target in $D'$,
\item\label{enu:edge--D'--v} $D'=D\cup\left\{ d\right\} $ is a subset
of $V\left(\varTheta\left(R\right)\right)$ of the least possible size
 such that for every vertex $v\in V\left(\varTheta\left(R\right)\right)\setminus D'$
there is an edge in $\varTheta\left(R\right)$ with target $v$ and
source in $D'$.
\end{enumerate}
\end{proposition}

\begin{proof}
Let $\mathscr{V}=\left\{ V_{1},V_{2},\ldots,V_{k}\right\} ,k=\binom{n}{1}_{q}$ be the set of all $1$-dimensional subspaces of $F_{q}^{n}$.
Since $n=2$, every vertex in $\varTheta\left(R\right)\setminus\left\{ \left[0\right],\left[1\right]\right\} $
is of the form $\left(V_{i},V_{j}\right)$ for some $i,j\in\left\{ 1,2,\ldots,k\right\} $.

Let $D'$ be any subset satisfying condition \ref{enu:edge--D'--v}. Using terminology from the proof
of Proposition~\ref{prop:dominating}, if for some $i\in\left\{ 1,2,\ldots,k\right\} $,
$V_{i}$ is not a kernel in $D'$ then for every $j\in\left\{ 1,2,\ldots,k\right\} $
there is no edge from $D'$ to $\left(V_{i},V_{j}\right)$. Hence,
$\left(V_{i},V_{j}\right)\in D'$ by assumption. In this case, $\left|D'\right|\geq\binom{n}{1}_{q}$. On the other hand, if $V_{i}$ is a kernel in $D'$ for all $i\in\left\{ 1,2,\ldots,k\right\} $,
then $\left|D'\right|\geq\binom{n}{1}_{q}$ as well. This establishes the bound for the size of the set in \ref{enu:edge--D'--v}. Similar argument works for \ref{enu:edge--v--D'}.

Let $D$ be any dominating set for $\overline{\varTheta}\left(R\right)\setminus\left\{ \left[0\right],\left[1\right]\right\}$. If $V_{i}$ is a kernel in $D$ for
all $i\in\left\{ 1,2,\ldots,k\right\} $ then $\left|D\right|\geq\binom{n}{1}_{q}$ . So suppose $V_{i}$ is not a kernel in $D$ for some $i$. As in
the proof of Proposition~\ref{prop:dominating}, for every $j\neq i$,
either $\left(V_{i},V_{j}\right)\in D$ or $\left(V_{j},V_{l_{j}}\right)\in D$
for some $l_{j}\in\left\{ 1,2,\ldots,k\right\} $. Since the sets
$\left\{ \left(V_{i},V_{j}\right),\left(V_{j},V_{l_{j}}\right)\right\}$, $j\in\left\{ 1,2,\ldots,k\right\} \setminus\left\{ i\right\} $,
are disjoint, we conclude that $\left|D\right|\geq k-1\geq\binom{n}{1}_{q}-1$.

Now, we construct $D$ and $d$ that satisfy all the conditions. Observe
that $k=\binom{n}{1}_{q}\geq3$. Hence, we can take $D=\left\{ \left(V_{1},V_{2}\right)\right\} \cup \left\{\left(V_{i},V_{i}\right):i\in\left\{ 3,4,\ldots,k\right\} \right\} $
and $d=\left(V_{2},V_{1}\right)$. Only $V_{1}$ is not a kernel in
$D$ and only $V_{2}$ is not an image in $D$. Hence, $\left(V_{1},V_{2}\right)$
is the only vertex not adjacent (in any direction) to any vertex in
$D$. But $\left(V_{1},V_{2}\right)\in D$, so $D$ is a dominating
set.
Every element of $\mathscr{V}$ is a kernel as well as an image in
$D'=D\cup\left\{ d\right\} $, hence $D'$ satisfies \ref{enu:edge--v--D'}
and \ref{enu:edge--D'--v}.
\end{proof}

\section{Graph-theoretic characterization of $\varTheta\left(M_{n}\left(F_{q}\right)\right)$}\label{sec:characterization}

Motivated by \cite[Propositions~4.1 and 4.2]{Dju-Jev-Sto}, the aim of this section is to give a purely graph-theoretic characterization of the graph $\varTheta\left(M_{n}\left(F_{q}\right)\right)$.

For a set $X \subseteq V(G)$, where $G$ is a directed graph, define 
\[
\cl_{\textrm{t}}\left(X\right)=N^{+}\left(N^{-}\left(X\right)\right) \quad\textup{and}\quad \cl_{\textrm{s}}\left(X\right)=N^{-}\left(N^{+}\left(X\right)\right).
\]

\begin{lemma}\label{lem:closure}
For a given directed graph $G$, $\cl_{\textrm{t}}$
(respectively $\cl_{\textrm{s}}$) is a closure operator on
$V\left(G\right)$, i.e. it satisfies the following properties for
all subsets $X,Y\subseteq V\left(G\right)$:
\begin{enumerate}
\item\label{enu:closure1} $X\subseteq\cl_{\textrm{t}}\left(X\right)$,
\item\label{enu:closure2} $X\subseteq Y$ implies $\cl_{\textrm{t}}\left(X\right)\subseteq\cl_{\textrm{t}}\left(Y\right)$,
\item\label{enu:closure3} $\cl_{\textrm{t}}\left(\cl_{\textrm{t}}\left(X\right)\right)=\cl_{\textrm{t}}\left(X\right)$.
\end{enumerate}
We will call $\cl_{\textrm{t}}\left(X\right)$ the \emph{target
closure} of $X$ and $\cl_{\textrm{s}}\left(X\right)$ the
\emph{source closure} of $X$.
\end{lemma}

\begin{proof}
\smallskip\noindent\ref{enu:closure1} The inclusion $X\subseteq\cl_{\textrm{t}}\left(X\right)$ follows directly from the definition of closure.

\smallskip\noindent\ref{enu:closure2} $X\subseteq Y$ implies $N^{-}\left(X\right)\supseteq N^{-}\left(Y\right)$
which further implies $N^{+}\left(N^{-}\left(X\right)\right)\subseteq N^{+}\left(N^{-}\left(Y\right)\right)$.

\smallskip\noindent\ref{enu:closure3} By property \ref{enu:closure1}, we have $X\subseteq\cl_{\textrm{t}}\left(X\right)$
and consequently $N^{-}\left(X\right)\supseteq N^{-}\left(\cl_{\textrm{t}}\left(X\right)\right)$.
By definition of $\cl_{\textrm{t}}\left(X\right)$, every
vertex in $N^{-}\left(X\right)$ is a common pre-neighbour of all
vertices in $\cl_{\textrm{t}}\left(X\right)$. Hence, $N^{-}\left(X\right)\subseteq N^{-}\left(\cl_{\textrm{t}}\left(X\right)\right)$.
We conclude that $N^{-}\left(X\right)=N^{-}\left(\cl_{\textrm{t}}\left(X\right)\right)$
which implies $N^{+}\left(N^{-}\left(X\right)\right)=N^{+}\left(N^{-}\left(\cl_{\textrm{t}}\left(X\right)\right)\right)$.
\end{proof}

We will use these properties of closure operators without further
reference.
A set $X$ will be called \emph{target-closed} (respectively \emph{source-closed})
if $\cl_{\textrm{t}}\left(X\right)=X$ (respectively $\cl_{\textrm{s}}\left(X\right)=X$).
Observe that an arbitrary intersection of target-closed sets is again
a target-closed set. Indeed, for two target-closed sets $A$ and $B$,
properties \ref{enu:closure2} and \ref{enu:closure3} in Lemma~\ref{lem:closure} imply $\cl_{\textrm{t}}\left(A\cap B\right)\subseteq\cl_{\textrm{t}}\left(A\right)=A$
and similarly for $B$. Hence, $\cl_{\textrm{t}}\left(A\cap B\right)\subseteq A\cap B$
which implies that $\cl_{\textrm{t}}\left(A\cap B\right)=A\cap B$
by property \ref{enu:closure1}. The proof for arbitrary intersections
is similar.

Given a directed graph $G$, let $d_{0}>d_{1}>d_{2}>\cdots>d_{n}$
be all the possible outdegrees of vertices of $G$. We will denote
by $V_{k}\left(G\right)$ the set of all vertices of $G$ whose outdegree
equals $d_{k}$. Clearly, the sets $V_{k}\left(G\right)$ form a partition
of $V\left(G\right)$. For $x\in V_{k}\left(G\right)$, we will say
that the \emph{type} of vertex $x$, denoted by $T(x)$, is equal to $k$.

In the next theorem we give a purely graph-theoretic characterization of the graph $\varTheta\left(M_{n}\left(F_{q}\right)\right)$. In the proof we will need the notion of projective spaces over finite fields. We denote the $(n-1)$-dimensional projective space over $F_q$ by $PG(n-1,q)$.
The points of $PG(n-1,q)$ are $1$-dimensional vector subspaces of $F_q^n$. A \emph{subspace} of $PG(n-1,q)$ is the set of points whose union is a vector subspace of $F_q^n$.
For further information on $PG(n-1,q)$ we refer the reader to Hirschfeld \cite{Hir} and Casse \cite{Cas}.

\begin{theorem}\label{prop:noncommutative--staircase}
Let $n\neq3$ be a positive
integer and $q>1$ a prime power. Up to a graph isomorphism, there
is a unique directed graph $G$ with no multiple edges that satisfies the
following five properties:
\begin{enumerate}
\item\label{enu:noncomm--sg1} $V\left(G\right)=\bigcup_{k=0}^{n}V_{k}\left(G\right)$, i.e. the number of different outdegrees in $G$ is $n+1$.
\item\label{enu:noncomm--sg2} $\left|V_{k}\left(G\right)\right|=\binom{n}{k}_{q}^{2}$ for all $k\in\left\{ 0,1,\ldots,n\right\} $.
\item\label{enu:noncomm--sg3} $\left|N^{+}\left(x\right)\cap V_{j}\left(G\right)\right|=\left|N^{-}\left(x\right)\cap V_{j}\left(G\right)\right|=\binom{n}{j}_{q}\binom{n-k}{j}_{q}$ for all $x\in V_{k}\left(G\right)$ and all $j,k\in\left\{ 0,1,\ldots,n\right\} $.
\item\label{enu:noncomm--sg4} Every target-closed set is a target closure
of a single vertex. Every source-closed set is a source closure of
a single vertex.
\item\label{enu:noncomm--sg5} $\left|\cl_{\textrm{t}}\left(x\right)\cap\cl_{\textrm{s}}\left(y\right)\cap V_{k}\left(G\right)\right|=1$
for all $x,y\in V_{k}\left(G\right)$ and all $k\in\left\{ 0,1,\ldots,n\right\} $.
\end{enumerate}
The graph $G$ is isomorphic to $\varTheta\left(M_{n}\left(F_{q}\right)\right)$.
\end{theorem}

We remark that property \ref{enu:noncomm--sg3} implies $\deg^{+}\left(x\right)=\deg^{-}\left(x\right)$
for every $x\in V\left(G\right)$, so the conditions ($i$)--($v$) are symmetric with respect to indegree and outdegree.

\begin{proof}
We divide the proof into several steps.

\begin{step}{1}
For all $x,y\in V_{k}\left(G\right)$, $\cl_{\textrm{t}}\left(x\right)\cap\cl_{\textrm{t}}\left(y\right)\cap V_{k}\left(G\right)\neq\emptyset$
implies $\cl_{\textrm{t}}\left(x\right)=\cl_{\textrm{t}}\left(y\right)$,
and $\cl_{\textrm{s}}\left(x\right)\cap\cl_{\textrm{s}}\left(y\right)\cap V_{k}\left(G\right)\neq\emptyset$
implies $\cl_{\textrm{s}}\left(x\right)=\cl_{\textrm{s}}\left(y\right)$.
\end{step}

Suppose $z\in\cl_{\textrm{t}}\left(x\right)\cap\cl_{\textrm{t}}\left(y\right)\cap V_{k}\left(G\right)$.
Then, by definition of $\cl_{\textrm{t}}$, $N^{-}\left(z\right)\supseteq N^{-}\left(x\right)\cup N^{-}\left(y\right)$.
This gives us
\[
N^{-}\left(z\right)\cap V_{j}\left(G\right)\supseteq\left(N^{-}\left(x\right)\cap V_{j}\left(G\right)\right)\cup\left(N^{-}\left(y\right)\cap V_{j}\left(G\right)\right)
\]
for all $j\in\left\{ 0,1,\ldots,n\right\}$.
Property \ref{enu:noncomm--sg3} implies that $\left|N^{-}\left(z\right)\cap V_{j}\left(G\right)\right|=\left|N^{-}\left(x\right)\cap V_{j}\left(G\right)\right|=\left|N^{-}\left(y\right)\cap V_{j}\left(G\right)\right|.$
Hence, $N^{-}\left(z\right)\cap V_{j}\left(G\right)=N^{-}\left(x\right)\cap V_{j}\left(G\right)=N^{-}\left(y\right)\cap V_{j}\left(G\right).$
Since this holds for all $j$, we have $N^{-}\left(x\right)=N^{-}\left(y\right)$
and consequently $\cl_{\textrm{t}}\left(x\right)=\cl_{\textrm{t}}\left(y\right)$.

\begin{step}{2}
$T\left(y\right)\leq T\left(x\right)$ for all $y\in\cl_{\textrm{t}}\left(x\right)$.
\end{step}

By property \ref{enu:noncomm--sg3}, $\left|N^{-}\left(x\right)\cap V_{n-T\left(x\right)}\left(G\right)\right|=\binom{n}{n-T\left(x\right)}_{q}\binom{n-T\left(x\right)}{n-T\left(x\right)}_{q}\neq0$, hence there exists $z\in N^{-}\left(x\right)\cap V_{n-T\left(x\right)}\left(G\right)$.
Since $y\in N^{+}\left(N^{-}\left(x\right)\right)$, we have $y\in N^{+}\left(z\right).$
Therefore, $\left|N^{+}\left(z\right)\cap V_{T\left(y\right)}\left(G\right)\right|\neq0$.
By property \ref{enu:noncomm--sg3} we have $\left|N^{+}\left(z\right)\cap V_{T\left(y\right)}\left(G\right)\right|=\binom{n}{T\left(y\right)}_{q}\binom{n-T\left(z\right)}{T\left(y\right)}_{q}$. The second factor is equal to $\binom{T\left(x\right)}{T\left(y\right)}_{q}$ by definition of $z$ and has to be nonzero. So, by definition of
$q$-binomial coefficients, we obtain $T\left(y\right)\leq T\left(x\right)$.

\begin{step}{3}
$y\in N^{+}\left(x\right)$ (equivalently, $x\in N^{-}\left(y\right)$)
implies $\cl_{\textrm{t}}\left(y\right)\subseteq N^{+}\left(x\right)$
and $\cl_{\textrm{s}}\left(x\right)\subseteq N^{-}\left(y\right)$.
\end{step}

By definition, $y\in N^{+}\left(x\right)$ implies $x\in N^{-}\left(y\right)$, which
further implies $N^{+}\left(x\right)\supseteq N^{+}\left(N^{-}\left(y\right)\right)=\cl_{\textrm{t}}\left(y\right)$.
The proof for source closure is similar.

\begin{step}{4}
$\left|\cl_{\textrm{t}}\left(x\right)\cap V_{k}\left(G\right)\right|=\left|\cl_{\textrm{t}}\left(y\right)\cap V_{k}\left(G\right)\right|$ for all $x,y\in V_{k}\left(G\right)$.
\end{step}

Define a map $f:\cl_{\textrm{t}}\left(x\right)\cap V_{k}\left(G\right)\to\cl_{\textrm{t}}\left(y\right)\cap V_{k}\left(G\right)$
as follows. For $z\in\cl_{\textrm{t}}\left(x\right)\cap V_{k}\left(G\right)$,
we use property \ref{enu:noncomm--sg5} to define $f\left(z\right)$
uniquely by the condition 
\[
\left\{ f\left(z\right)\right\} =\cl_{\textrm{t}}\left(y\right)\cap\cl_{\textrm{s}}\left(z\right)\cap V_{k}\left(G\right).
\]
Similarly, define $g:\cl_{\textrm{t}}\left(y\right)\cap V_{k}\left(G\right)\to\cl_{\textrm{t}}\left(x\right)\cap V_{k}\left(G\right)$
by 
\[
\left\{ g\left(w\right)\right\} =\cl_{\textrm{t}}\left(x\right)\cap\cl_{\textrm{s}}\left(w\right)\cap V_{k}\left(G\right).
\]
It suffices to prove that maps $f$ and $g$ are inverses of each other.
Choose any $z\in\cl_{\textrm{t}}\left(x\right)\cap V_{k}\left(G\right)$.
Since $f\left(z\right)\in\cl_{\textrm{s}}\left(f\left(z\right)\right)\cap\cl_{\textrm{s}}\left(z\right)\cap V_{k}\left(G\right)$,
Step~1 implies $\cl_{\textrm{s}}\left(f\left(z\right)\right)=\cl_{\textrm{s}}\left(z\right)$.
Hence $z\in\cl_{\textrm{t}}\left(x\right)\cap\cl_{\textrm{s}}\left(f\left(z\right)\right)\cap V_{k}\left(G\right)=\left\{ g\left(f\left(z\right)\right)\right\} $.
We conclude that $g\left(f\left(z\right)\right)=z$. Similarly, $f\left(g\left(w\right)\right)=w$.

\begin{step}{5}
$\left|\cl_{\textrm{t}}\left(x\right)\cap V_{k}\left(G\right)\right|=\left|\cl_{\textrm{s}}\left(x\right)\cap V_{k}\left(G\right)\right|=\binom{n}{k}_{q}$ for all $x\in V_{k}\left(G\right)$.
\end{step}

Let $z\in V_{n-k}\left(G\right)$. If $y\in N^{+}\left(z\right)\cap V_{k}\left(G\right)$,
then $\cl_{\textrm{t}}\left(y\right)\cap V_{k}\left(G\right)\subseteq N^{+}\left(z\right)\cap V_{k}\left(G\right)$
by Step~3, so $N^{+}\left(z\right)\cap V_{k}\left(G\right)$ is a
union of sets of the form $\cl_{\textrm{t}}\left(y\right)\cap V_{k}\left(G\right)$
where $y\in N^{+}\left(z\right)\cap V_{k}\left(G\right)$. By Step~1, two sets of this form are either disjoint or equal and, by Step~4, they all have the same size. Hence, $\left|\cl_{\textrm{t}}\left(x\right)\cap V_{k}\left(G\right)\right|$
divides $\left|N^{+}\left(z\right)\cap V_{k}\left(G\right)\right|$, which is equal to $\binom{n}{k}_{q}\binom{k}{k}_{q}=\binom{n}{k}_{q}$ by property \ref{enu:noncomm--sg3}. Similarly, $\left|\cl_{\textrm{s}}\left(x\right)\cap V_{k}\left(G\right)\right|$
divides $\binom{n}{k}_{q}$.

By Step~1, the set $V_{k}\left(G\right)$ is a disjoint union of some
number of sets of the form $\cl_{\textrm{t}}\left(w\right)\cap V_{k}\left(G\right)$,
where $w\in V_{k}\left(G\right)$. Since each set of this form intersects
$\cl_{\textrm{s}}\left(x\right)\cap V_{k}\left(G\right)$
in precisely one element by property \ref{enu:noncomm--sg5}, $V_{k}\left(G\right)$
is a disjoint union of $\left|\cl_{\textrm{s}}\left(x\right)\cap V_{k}\left(G\right)\right|$
sets of the form $\cl_{\textrm{t}}\left(w\right)\cap V_{k}\left(G\right)$.
Using Step~4 and property \ref{enu:noncomm--sg2}, we obtain $\left|\cl_{\textrm{t}}\left(x\right)\cap V_{k}\left(G\right)\right|\cdot\left|\cl_{\textrm{s}}\left(x\right)\cap V_{k}\left(G\right)\right|=\left|V_{k}\left(G\right)\right|=\binom{n}{k}_{q}^{2}$. Both factors on the left-hand side divide $\binom{n}{k}_{q}$, hence $\left|\cl_{\textrm{t}}\left(x\right)\cap V_{k}\left(G\right)\right|=\left|\cl_{\textrm{s}}\left(x\right)\cap V_{k}\left(G\right)\right|=\binom{n}{k}_{q}$.

\begin{step}{6}
$\left|N^{+}\left(x\right)\cap N^{-}\left(y\right)\cap V_{n-k}\left(G\right)\right|=1$ for all $x,y\in V_{k}\left(G\right)$.
\end{step}

By property \ref{enu:noncomm--sg3}, there are $x',y'\in V_{n-k}\left(G\right)$
such that $x'\in N^{+}\left(x\right)$ and $y'\in N^{-}\left(y\right)$.
Using property \ref{enu:noncomm--sg5}, let
\[
\left\{ z\right\} =\cl_{\textrm{t}}\left(x'\right)\cap\cl_{\textrm{s}}\left(y'\right)\cap V_{n-k}\left(G\right).
\]

By definition of closures, it follows that $z\in N^{+}\left(x\right)\cap N^{-}\left(y\right)$.
To prove uniqueness of $z$, suppose $w\in N^{+}\left(x\right)\cap N^{-}\left(y\right)\cap V_{n-k}\left(G\right)$.
By Step~3, we have $\cl_{\textrm{t}}\left(w\right)\cap V_{n-k}\left(G\right)\subseteq N^{+}\left(x\right)\cap V_{n-k}\left(G\right)$.
But, using Step~5 and property \ref{enu:noncomm--sg3}, we get $\left|\cl_{\textrm{t}}\left(w\right)\cap V_{n-k}\left(G\right)\right|=\binom{n}{n-k}_{q}=\left|N^{+}\left(x\right)\cap V_{n-k}\left(G\right)\right|$. Thus 
\[
\cl_{\textrm{t}}\left(w\right)\cap V_{n-k}\left(G\right)=N^{+}\left(x\right)\cap V_{n-k}\left(G\right).
\]

Similarly $\cl_{\textrm{s}}\left(w\right)\cap V_{n-k}\left(G\right)=N^{-}\left(y\right)\cap V_{n-k}\left(G\right)$.
By replacing $w$ with $z$ in the last two equations, we deduce $\cl_{\textrm{t}}\left(z\right)\cap V_{n-k}\left(G\right)=\cl_{\textrm{t}}\left(w\right)\cap V_{n-k}\left(G\right)$
and $\cl_{\textrm{s}}\left(z\right)\cap V_{n-k}\left(G\right)=\cl_{\textrm{s}}\left(w\right)\cap V_{n-k}\left(G\right)$.
By property \ref{enu:noncomm--sg5}, we conclude
\[
\left\{ z\right\} =\cl_{\textrm{t}}\left(z\right)\cap\cl_{\textrm{s}}\left(z\right)\cap V_{n-k}\left(G\right)=\cl_{\textrm{t}}\left(w\right)\cap\cl_{\textrm{s}}\left(w\right)\cap V_{n-k}\left(G\right)=\left\{ w\right\},
\]
which completes the proof of Step~6.

\smallskip For any $v\in V_{k}\left(G\right)$, Step~6 allows us to define $v^{\textrm{op}}$, the \emph{opposite vertex} of $v$,
by the condition 
\[
\left\{ v^{\textrm{op}}\right\} =N^{+}\left(v\right)\cap N^{-}\left(v\right)\cap V_{n-k}\left(G\right).
\]
Observe that $\left(v^{\textrm{op}}\right)^{\textrm{op}}=v$ since
$v\in N^{+}\left(v^{\textrm{op}}\right)\cap N^{-}\left(v^{\textrm{op}}\right)\cap V_{k}\left(G\right)$.
In addition, the proof of Step~6 shows that
\[
\cl_{\textrm{t}}\left(v^{\textrm{op}}\right)\cap V_{n-k}\left(G\right)=N^{+}\left(v\right)\cap V_{n-k}\left(G\right).
\]
Replacing $v$ by $v^{\textrm{op}}$, we obtain
\begin{equation}\label{eq:closure--neighbourhood}
\cl_{\textrm{t}}\left(v\right)\cap V_{k}\left(G\right)=N^{+}\left(v^{\textrm{op}}\right)\cap V_{k}\left(G\right).
\end{equation}

\begin{step}{7}
$\cl_{\textrm{t}}\left(x\right)=N^{+}\left(x^{\textrm{op}}\right)$
and $\cl_{\textrm{s}}\left(x\right)=N^{-}\left(x^{\textrm{op}}\right)$ for all $x \in V(G)$.
\end{step}

Let $x\in V_{k}\left(G\right)$. By Step~3 and the definition of $x^{\textrm{op}}$,
$\cl_{\textrm{t}}\left(x\right)\subseteq N^{+}\left(x^{\textrm{op}}\right)$.
To prove the opposite inclusion, let $z\in N^{+}\left(x^{\textrm{op}}\right)$.
By property \ref{enu:noncomm--sg4}, we have $\cl_{\textrm{t}}\left(\left\{ x,z\right\} \right)=\cl_{\textrm{t}}\left(u\right)$
for some $u\in V_{j}\left(G\right)$, $j\in\left\{ 0,1,\ldots,n\right\} $.
Since $x^{\textrm{op}}\in N^{-}\left(\left\{ x,z\right\} \right)$
and $u\in\cl_{\textrm{t}}\left(u\right)=N^{+}\left(N^{-}\left(\left\{ x,z\right\} \right)\right)$,
we have $u\in N^{+}\left(x^{\textrm{op}}\right)$. By property \ref{enu:noncomm--sg3}
(using similar argument as in Step~2), we obtain $j\leq k$. Since
$x\in\cl_{\textrm{t}}\left(u\right)$, Step~2 implies $k\leq j$.
Hence, $j=k.$ By \eqref{eq:closure--neighbourhood}, we have
\[
\cl_{\textrm{t}}\left(x\right)\cap V_{k}\left(G\right)=N^{+}\left(x^{\textrm{op}}\right)\cap V_{k}\left(G\right)\ni u.
\]
This implies $u\subseteq\cl_{\textrm{t}}\left(x\right)$,
hence $\cl_{\textrm{t}}\left(x\right)=\cl_{\textrm{t}}\left(u\right)\ni z$.
We conclude that $N^{+}\left(x^{\textrm{op}}\right)\subseteq\cl_{\textrm{t}}\left(x\right)$.

\smallskip In the next few steps, we will prove several properties of the target-closed
sets which will enable us to introduce a projective space isomorphic
to $PG\left(n-1,q\right)$. We will make use of the characterization
given in \cite[p.~39]{Hir}.

Let $\mathscr{P}_{k}$ be the set of all $\cl_{\textrm{t}}\left(x\right)$,
where $x\in V_{k}\left(G\right)$. By Step~2, the sets $\mathscr{P}_{k}$
for different indices $k$ are disjoint. Elements of $\mathscr{P}_{1}$
will be called \emph{points}, and elements of $\mathscr{P}{}_{2}$ will be
called \emph{lines}, and we will be using all the standard geometric terminology
such as point lies on a line, two lines intersect, etc. 

\begin{step}{8}
Every element of $\mathscr{P}_{m}$ contains $\binom{m}{k}_{q}$ elements of $\mathscr{P}_{k}$ and is contained in $\binom{n-m}{n-k}_{q}$ elements of $\mathscr{P}_{k}$. In particular, $\left|\mathscr{P}_{k}\right|=\binom{n}{k}_{q}$.
\end{step}

Let $x\in V_{m}\left(G\right)$ and $y\in V_{k}\left(G\right)$. Then,
by Step~7, $\cl_t\left(y\right)\subseteq\cl_{\textrm{t}}\left(x\right)$
is equivalent to $\cl_t\left(y\right)\subseteq N^{+}\left(x^{\textrm{op}}\right)$
which is, by Step~3, further equivalent to $y\in N^{+}\left(x^{\textrm{op}}\right)$
and this is, again by Step~3, equivalent to $\cl_{\textrm{t}}\left(y\right)\cap V_{k}\left(G\right)\subseteq N^{+}\left(x^{\textrm{op}}\right)\cap V_{k}\left(G\right)$.
In addition, it follows from Step~1 and Step~3 that $N^{+}\left(x^{\textrm{op}}\right)\cap V_{k}\left(G\right)$
is a disjoint union of sets of the form $\cl_{\textrm{t}}\left(z\right)\cap V_{k}\left(G\right)$,
where $z\in N^{+}\left(x^{\textrm{op}}\right)\cap V_{k}\left(G\right)$.
Hence, using Step~5 and property \ref{enu:noncomm--sg3}, we see that
$\cl_{\textrm{t}}\left(x\right)$ contains
\[
\frac{\left|N^{+}\left(x^{\textrm{op}}\right)\cap V_{k}\left(G\right)\right|}{\left|\cl_{\textrm{t}}\left(y\right)\cap V_{k}\left(G\right)\right|}=\frac{\binom{n}{k}_{q}\binom{m}{k}_{q}}{\binom{n}{k}_{q}}=\binom{m}{k}_{q}
\]
elements of $\mathscr{P}_{k}$. To prove the second part, observe
that, by Step~7, $\cl_{\textrm{t}}\left(x\right)\subseteq\cl_{\textrm{t}}\left(y\right)$
is equivalent to $x\in N^{+}\left(y^{\textrm{op}}\right)$
which is equivalent to $y^{\textrm{op}}\in N^{-}\left(x\right)$ and
this is, again by Step~7, equivalent to $\cl_{\textrm{s}}\left(y^{\textrm{op}}\right)\subseteq\cl_{\textrm{s}}\left(x^{\textrm{op}}\right)$.
By symmetry with the above and the fact that $T\left(x^{\textrm{op}}\right)=n-m$
and $T\left(y^{\textrm{op}}\right)=n-k$, we conclude that $\cl_{\textrm{t}}\left(x\right)$
is contained in $\binom{n-m}{n-k}_{q}$ elements of $\mathscr{P}_{k}$.

To prove $\left|\mathscr{P}_{k}\right|=\binom{n}{k}_{q}$, let $x\in V_{n}\left(G\right)$. By Step~7, we have $\cl_{\textrm{t}}\left(x\right)=N^{+}\left(x^{\textrm{op}}\right)$,
where $x^{\textrm{op}}\in V_{0}\left(G\right)$. For every $j\in\left\{ 0,1,\ldots,n\right\} $,
properties \ref{enu:noncomm--sg2} and \ref{enu:noncomm--sg3} give
us
\[
\left|N^{+}\left(x^{\textrm{op}}\right)\cap V_{j}\left(G\right)\right|=\binom{n}{j}_{q}^{2}=\left|V_{j}\left(G\right)\right|,
\]
which implies $V_{j}\left(G\right)\subseteq N^{+}\left(x^{\textrm{op}}\right)$.
By property \ref{enu:noncomm--sg1}, $\cl_{\textrm{t}}\left(x\right)=V\left(G\right)$.
Now, $\left|\mathscr{P}_{k}\right|=\binom{n}{k}_{q}$ follows from the first part of Step~8.

\smallskip Observe that in our situation, an element of $\mathscr{P}_{k}$ is
not a union of points it contains, but rather their closure, as
we show next.

\begin{step}{9}
$\cl_{\textrm{t}}\left(x\right)=\cl_{\textrm{t}}\left(\bigcup\left\{ P\in\mathscr{P}_{1}:P\subseteq\cl_{\textrm{t}}\left(x\right)\right\} \right)$ for all $x\in V\left(G\right)$.
\end{step}

By property \ref{enu:noncomm--sg4}, $\cl_{\textrm{t}}\left(\bigcup\left\{ P\in\mathscr{P}_{1}:P\subseteq\cl_{\textrm{t}}\left(x\right)\right\} \right)=\cl_{\textrm{t}}\left(u\right)$
for some $u\in V\left(G\right)$. By Step~8, $\cl_{\textrm{t}}\left(x\right)$
contains $\binom{T\left(x\right)}{1}_{q}$ points and $\cl_{\textrm{t}}\left(u\right)$ contains $\binom{T\left(u\right)}{1}_{q}$ points. By construction, $\cl_{\textrm{t}}\left(u\right)$
contains all the points of $\cl_{\textrm{t}}\left(x\right)$,
so $\binom{T\left(u\right)}{1}_{q}\geq\binom{T\left(x\right)}{1}_{q}$. On the other hand, $\binom{T\left(u\right)}{1}_{q}\leq\binom{T\left(x\right)}{1}_{q}$ since $\cl_{\textrm{t}}\left(u\right)\subseteq\cl_{\textrm{t}}\left(x\right)$.
Hence, $\binom{T\left(u\right)}{1}_{q}=\binom{T\left(x\right)}{1}_{q}$ and consequently $T\left(u\right)=T\left(x\right)$. Thus, Step~1
implies $\cl_{\textrm{t}}\left(u\right)=\cl_{\textrm{t}}\left(x\right)$.

\begin{step}{10}
For $X\in\mathscr{P}_{k}$ and $Y\in\mathscr{P}_{1}$,
where $Y\nsubseteq X$, there is at most one $Z\in\mathscr{P}_{k+1}$
such that $X,Y\subseteq Z$.
\end{step}

Suppose $Z,Z'\in\mathscr{P}_{k+1}$ both contain $X$ and $Y$. Then
$\cl_{\textrm{t}}\left(X\cup Y\right)\subseteq Z\cap Z'$.
Since $Z\cap Z'$ is a target-closed set, we have $Z\cap Z'=\cl_{\textrm{t}}\left(u\right)$
for some $u\in V_{j}\left(G\right)$ by property \ref{enu:noncomm--sg4}.
By Step~2, we have $k\leq j$ since $X\subseteq\cl_{\textrm{t}}\left(u\right)$.
If $k=j$, then $X\subseteq\cl_{\textrm{t}}\left(u\right)$
would imply $X=\cl_{\textrm{t}}\left(u\right)\supseteq Y$
by Step~1, a contradiction. Hence $k<j$. By Step~2, we also have
$j\leq k+1$ since $u\in Z$. Thus, $j=k+1$. So, $u\in Z\cap Z'\cap V_{k+1}\left(G\right)$
which implies $Z=Z'$ by Step~1.

\begin{step}{11}
For $X\in\mathscr{P}_{k}$ and $Y\in\mathscr{P}_{1}$,
where $Y\nsubseteq X$, there is precisely one $Z\in\mathscr{P}_{k+1}$
such that $X,Y\subseteq Z$.
\end{step}

First suppose $k=1$. Let us count the number of pairs $\left\{ X,Y\right\} $,
where $X\in\mathscr{P}_{1}$, $Y\in\mathscr{P}_{1}$ and $Y\nsubseteq X$,
such that $X,Y\subseteq Z$ for some $Z\in\mathscr{P}_{2}$. For convenience,
we will simply say that a pair $\left\{ X,Y\right\} $ is contained
in $Z$ if $X,Y\subseteq Z$. Observe that, by Step~1, the condition
$Y\nsubseteq X$ is equivalent to $Y\neq X$. By Step~8, every element
of $\mathscr{P}_{2}$ contains $\binom{2}{1}_{q}$ elements of $\mathscr{P}_{1}$. Hence, every $Z\in\mathscr{P}_{2}$
contains $\dbinom{\binom{2}{1}_{q}}{2}$ such pairs $\left\{ X,Y\right\}$. But, by Step~10, two distinct elements
of $\mathscr{P}_{2}$ cannot contain the same pair $\left\{ X,Y\right\} $.
Hence, by the equality $\left|\mathscr{P}_{2}\right|=\binom{n}{2}_{q}$ in Step~8, there are exactly $\dbinom{\binom{2}{1}_{q}}{2}\cdot\binom{n}{2}_{q}$ such pairs $\left\{ X,Y\right\} $. A short calculation shows that
$\dbinom{\binom{2}{1}_{q}}{2}\cdot\binom{n}{2}_{q}=\dbinom{\binom{n}{1}_{q}}{2}$, but, by Step~8, this is precisely the number of pairs of distinct
elements of $\mathscr{P}_{1}$. This shows that every pair of distinct
elements of $\mathscr{P}_{1}$ is contained in some $Z\in\mathscr{P}_{2}$
and, by Step~10, this $Z$ is unique.

Now suppose $k>1$. First, let us count the number of pairs $\left\{ X,Y\right\} $,
where $X\in\mathscr{P}_{k}$, $Y\in\mathscr{P}_{1}$ and $Y\nsubseteq X$,
such that $X,Y\subseteq Z$ for some $Z\in\mathscr{P}_{k+1}$. We
fix some $Z\in\mathscr{P}_{k+1}$ and count such pairs contained in
$Z$, as follows. By Step~8, we first choose $X$ in $\binom{k+1}{k}_{q}=\binom{k+1}{1}_{q}$ ways and then $Y$ in $\binom{k+1}{1}_{q}-\binom{k}{1}_{q}$ ways (the number of points in $Z$ minus the number of points in $X$).
Observe that no pair is counted twice since $k>1$. So there are $\binom{k+1}{1}_{q}\cdot\left(\binom{k+1}{1}_{q}-\binom{k}{1}_{q}\right)$ such pairs in a given $Z\in\mathscr{P}_{k+1}$. By Step~10, each
pair $\left\{ X,Y\right\} $ is contained in only one $Z\in\mathscr{P}_{k+1}$.
Hence, by the equality $\left|\mathscr{P}_{k+1}\right|=\binom{n}{k+1}_{q}$ in Step~8, there are exactly $\binom{k+1}{1}_{q}\cdot\left(\binom{k+1}{1}_{q}-\binom{k}{1}_{q}\right)\cdot\binom{n}{k+1}_{q}$ pairs $\left\{ X,Y\right\} $, where $X\in\mathscr{P}_{k}$, $Y\in\mathscr{P}_{1}$
and $Y\nsubseteq X$, that are contained in some $Z$.

Now, let us count the number of pairs $\left\{ X,Y\right\} $, where
$X\in\mathscr{P}_{k}$, $Y\in\mathscr{P}_{1}$ and $Y\nsubseteq X$
(here, we do not insist that $X$ and $Y$ should be contained in
some $Z\in\mathscr{P}_{k+1}$). Similar as above, we first choose
$X\in\mathscr{P}_{k}$ and then $Y\in\mathscr{P}_{1}$, $Y\nsubseteq X$,
to count that there are exactly $\binom{n}{k}_{q}\cdot\left(\binom{n}{1}_{q}-\binom{k}{1}_{q}\right)$ such pairs $\left\{ X,Y\right\} $. A short calculation shows that
$\binom{k+1}{1}_{q}\cdot\left(\binom{k+1}{1}_{q}-\binom{k}{1}_{q}\right)\cdot\binom{n}{k+1}_{q}=\binom{n}{k}_{q}\cdot\left(\binom{n}{1}_{q}-\binom{k}{1}_{q}\right)$, which implies that each pair $\left\{ X,Y\right\} $, where $X\in\mathscr{P}_{k}$,
$Y\in\mathscr{P}_{1}$ and $Y\nsubseteq X$, is in fact contained
in some $Z\in\mathscr{P}_{k+1}$.

\begin{step}{12}
If $X\in\mathscr{P}_{k}$ and $Y\in\mathscr{P}_{1}$,
where $Y\nsubseteq X$, then $\cl_{\textrm{t}}\left(X\cup Y\right)\in\mathscr{P}_{k+1}$.
\end{step}

By property \ref{enu:noncomm--sg4}, $\cl_{\textrm{t}}\left(X\cup Y\right)=\cl_{\textrm{t}}\left(u\right)$
for some $u\in V_{j}\left(G\right)$ where $j\in\left\{ 0,1,\ldots,n\right\} $.
By Step~11, $\cl_{\textrm{t}}\left(u\right)\subseteq\cl_{\textrm{t}}\left(v\right)$
for some $v\in V_{k+1}\left(G\right)$. By Step~2, we have $k\leq j\leq k+1$.
If $j=k$, then the inclusion $X\subseteq\cl_{\textrm{t}}\left(u\right)$
implies $X=\cl_{\textrm{t}}\left(u\right)\supseteq Y$ by
Step~1, a contradiction. Hence, $j=k+1$.

\begin{step}{13}
For all $X\in\mathscr{P}_{k}$, $Y\in\mathscr{P}_{2}$
and $Z\in\mathscr{P}_{k+1}$, where $X,Y\subseteq Z$, there exists
$W\in\mathscr{P}_{1}$ such that $W\subseteq X\cap Y$.
\end{step}

We count the number of lines that are contained in $Z$ and contain a common
point with $X$. Step~8 implies that every line contains $\binom{2}{1}_{q}\geq3$
distinct points. Step~11 says, in particular, that a line is determined
uniquely by two distinct points it contains. We will use this fact
implicitly in our counting. We first count those lines contained in
$Z$, that have only one common point with $X$, as follows. We choose
a point contained in $X$ in $\binom{k}{1}_{q}$ ways and then a point
contained in $Z\setminus X$ in $\binom{k+1}{1}_{q}-\binom{k}{1}_{q}$
ways. These two points determine a line that has only one common point
with $X$ (if a line contained two distinct points in $X$, then the
whole line would be contained in $X$), but this way every line is
counted $\binom{2}{1}_{q}-1$ times (this is the number of points
on the line, that are not contained in $X$). So there are
\[
\frac{\binom{k}{1}_{q}\cdot\left(\binom{k+1}{1}_{q}-\binom{k}{1}_{q}\right)}{\binom{2}{1}_{q}-1}=q^{k-1}\cdot\binom{k}{1}_{q}
\]
lines that are contained in $Z$ and have only one common point with
$X$. Any line that has more than one common point with $X$, is contained
in $X$, and there are $\binom{k}{2}_{q}$ such lines by Step~8. So,
using the recursion formula \eqref{eq:rekurzija} for $q$-binomial coefficients, we see
that altogether there are $\binom{k}{2}_{q}+q^{k-1}\cdot\binom{k}{1}_{q}=\binom{k+1}{2}_{q}$
lines contained in $Z$ that have a common point with $X$. However,
by Step~8, there are only $\binom{k+1}{2}_{q}$ lines contained in
$Z$ altogether, so every line contained in $Z$ has a common point with $X$.

\begin{step}{14}
If $\cl_{\textrm{t}}\left(\left\{ x,y\right\} \right)=\cl_{\textrm{t}}\left(z\right)$
and $\cl_{\textrm{t}}\left(x\right)\cap\cl_{\textrm{t}}\left(y\right)=\cl_{\textrm{t}}\left(w\right)$
then $T\left(x\right)+T\left(y\right)=T\left(z\right)+T\left(w\right)$.
\end{step}

It follows from Step~2 that $T\left(w\right)\leq T\left(y\right)$.
Suppose $T\left(w\right)=T\left(y\right)$. Then $\cl_{\textrm{t}}\left(w\right)=\cl_{\textrm{t}}\left(y\right)$
by Step~1. Consequently, $\cl_{\textrm{t}}\left(y\right)\subseteq\cl_{\textrm{t}}\left(x\right)$,
hence $\cl_{\textrm{t}}\left(x\right)=\cl_{\textrm{t}}\left(\left\{ x,y\right\} \right)=\cl_{\textrm{t}}\left(z\right)$
and $T\left(x\right)=T\left(z\right)$. In this
case, the claim holds.

We continue by induction on $T\left(y\right)-T\left(w\right)$. So
suppose $T\left(y\right)-T\left(w\right)\geq1$. Then, by Step~8, $\cl_{\textrm{t}}\left(y\right)$
contains more points than $\cl_{\textrm{t}}\left(w\right)$,
so we can choose a vertex $p\in V_{1}\left(G\right)$, such that $p\in\cl_{\textrm{t}}\left(y\right)\setminus\cl_{\textrm{t}}\left(w\right)$.
By Step~12, we have
\begin{align*}
\cl_{\textrm{t}}\left(\left\{ x,p\right\} \right) &=\cl_{\textrm{t}}\left(\cl_{\textrm{t}}\left(x\right)\cup\cl_{\textrm{t}}\left(p\right)\right)=\cl_{\textrm{t}}\left(x'\right) \quad\textup{and}\\
\cl_{\textrm{t}}\left(\left\{ w,p\right\} \right) &=\cl_{\textrm{t}}\left(\cl_{\textrm{t}}\left(w\right)\cup\cl_{\textrm{t}}\left(p\right)\right)=\cl_{\textrm{t}}\left(w'\right),
\end{align*}
where $T\left(x'\right)=T\left(x\right)+1$ and $T\left(w'\right)=T\left(w\right)+1$.
We will show that
\begin{align*}
\cl_{\textrm{t}}\left(\left\{ x',y\right\} \right) &=\cl_{\textrm{t}}\left(z\right) \quad\textup{and}\\
\cl_{\textrm{t}}\left(x'\right)\cap\cl_{\textrm{t}}\left(y\right) &=\cl_{\textrm{t}}\left(w'\right).
\end{align*}
Clearly, we have $\cl_{\textrm{t}}\left(\left\{ x',y\right\} \right)\subseteq\cl_{\textrm{t}}\left(z\right)$, and in addition, $\cl_{\textrm{t}}\left(z\right)=\cl_{\textrm{t}}\left(\left\{ x,y\right\} \right)\subseteq\cl_{\textrm{t}}\left(\left\{ x',y\right\} \right)$.
So, $\cl_{\textrm{t}}\left(\left\{ x',y\right\} \right)=\cl_{\textrm{t}}\left(z\right)$.
Take any $r\in\cl_{\textrm{t}}\left(x'\right)\cap\cl_{\textrm{t}}\left(y\right)$
such that $r\notin\cl_{\textrm{t}}\left(p\right)$. Then $\cl_{\textrm{t}}\left(\left\{ p,r\right\} \right)\in\mathscr{P}_{2}$
by Step~12. Since  $\cl_{\textrm{t}}\left(x\right)$ and $\cl_{\textrm{t}}\left(\left\{ p,r\right\} \right)$ are subsets of $\cl_{\textrm{t}}\left(x'\right)$,
by Step~13, there is $s\in V_{1}\left(G\right)$ such that $\cl_{\textrm{t}}\left(s\right)\subseteq\cl_{\textrm{t}}\left(\left\{ p,r\right\} \right)\cap\cl_{\textrm{t}}\left(x\right)$.
Since $p,r\in\cl_{\textrm{t}}\left(y\right)$, we have $\cl_{\textrm{t}}\left(s\right)\subseteq\cl_{\textrm{t}}\left(w\right)\subseteq\cl_{\textrm{t}}\left(w'\right)$.
In particular, $p\notin\cl_{\textrm{t}}\left(s\right)$ since
$p\notin\cl_{\textrm{t}}\left(w\right)$. Hence, $\cl_{\textrm{t}}\left(\left\{ p,s\right\} \right)\in\mathscr{P}_{2}$
by Step~12. But $\cl_{\textrm{t}}\left(\left\{ p,s\right\} \right)\subseteq\cl_{\textrm{t}}\left(\left\{ p,r\right\} \right)$,
and thus $\cl_{\textrm{t}}\left(\left\{ p,s\right\} \right)=\cl_{\textrm{t}}\left(\left\{ p,r\right\} \right)$
by Step~1. This implies $r\in\cl_{\textrm{t}}\left(\left\{ p,s\right\} \right)\subseteq\cl_{\textrm{t}}\left(w'\right)$
by the above. Since $\cl_{\textrm{t}}\left(p\right)\subseteq\cl_{\textrm{t}}\left(w'\right)$,
we have thus shown that $\cl_{\textrm{t}}\left(x'\right)\cap\cl_{\textrm{t}}\left(y\right)\subseteq\cl_{\textrm{t}}\left(w'\right)$.
On the other hand, $\cl_{\textrm{t}}\left(w'\right)=\cl_{\textrm{t}}\left(\left\{ w,p\right\} \right)\subseteq\cl_{\textrm{t}}\left(x'\right)\cap\cl_{\textrm{t}}\left(y\right)$,
hence $\cl_{\textrm{t}}\left(x'\right)\cap\cl_{\textrm{t}}\left(y\right)=\cl_{\textrm{t}}\left(w'\right)$.
Since $T\left(y\right)-T\left(w'\right)=T\left(y\right)-T\left(w\right)-1<T\left(y\right)-T\left(w\right)$,
it follows by induction that $T\left(x'\right)+T\left(y\right)=T\left(z\right)+T\left(w'\right)$
and thus $T\left(x\right)+T\left(y\right)=T\left(z\right)+T\left(w\right)$.

\smallskip Now, we introduce a projective space $\mathscr{S}_{n-1}$ following
a characterization in \cite[p.~39]{Hir} (see also \cite{Cas}). The set
of points of our projective space $\mathscr{S}_{n-1}$ will be $\mathscr{P}_{1}$.
For any $X\in\mathscr{P}_{k}$, $k\in\left\{ 0,1,\ldots,n\right\} $,
we introduce a subspace
\[
S\left(X\right)=\left\{ P\in\mathscr{P}_{1}:P\subseteq X\right\} 
\]
and define its (projective) dimension to be $\dim \left(S\left(X\right)\right)=k-1$.
Note that, by Step~9, $S$ is a bijection from the set of all target-closed
sets to the set of all subspaces of $\mathscr{P}_{1}$, and since the
sets $\mathscr{P}_{k}$ for different indices $k$ are disjoint, the
dimension is well-defined.

\begin{step}{15}
For $n\geq4$, $\mathscr{S}_{n-1}$ is a projective space isomorphic to $PG\left(n-1,q\right)$.
\end{step}

We will show that the space $\mathscr{S}_{n-1}$ satisfies the axioms for projective
space $PG\left(n-1,q\right)$ given in \cite[p.~39]{Hir}. Clearly, the set of all possible dimensions
of subspaces is $\left\{ -1,0,\ldots,n-1\right\} $. By property \ref{enu:noncomm--sg2},
there is unique subspace of dimension $-1$, and by Step~2, it is precisely
the empty set. Observe that subspaces of dimension $0$ are precisely
points. By property \ref{enu:noncomm--sg2}, there is a unique subspace
of dimension $n-1$, and by the proof of Step~8, it is precisely the whole $\mathscr{P}_{1}$.
If $S\left(X\right)\subseteq S\left(Y\right)$ then $X\subseteq Y$
by Step~9, which implies $\dim \left(S\left(X\right)\right)\leq\dim \left(S\left(Y\right)\right)$
by Step~2. Here, $S\left(X\right)=S\left(Y\right)$ if and only if
$X=Y$ which is equivalent to $\dim \left(S\left(X\right)\right)=\dim \left(S\left(Y\right)\right)$
by Step~1. The intersection of two subspaces is again a subspace.
This follows from the fact that target-closed sets are closed for
intersections and from property \ref{enu:noncomm--sg4}. Given
two subspaces $S\left(X\right)$ and $S\left(Y\right)$, the span
of $S\left(X\right)$ and $S\left(Y\right)$, i.e. the intersection
of all the subspaces containing both, is clearly $S\left(\cl_{\textrm{t}}\left(X\cup Y\right)\right)$
by Step~9. Step~14 easily implies
\[
\dim S\left(X\right)+\dim S\left(Y\right)=\dim \left(S\left(X\right)\cap S\left(Y\right)\right)+\dim S\left(\cl_{\textrm{t}}\left(X\cup Y\right)\right)
\]
because $S\left(X\right)\cap S\left(Y\right)=S\left(X\cap Y\right)$.
Finally, by Step~8, every subspace of dimension $1$ contains $\binom{2}{1}_{q}=q+1\geq3$
points.

\begin{step}{16}
For $n \leq 2$, $\mathscr{S}_{n-1}$ is a projective space isomorphic to $PG\left(n-1,q\right)$.
\end{step}

If $n=1$, then the projective space $\mathscr{S}_{n-1}$ contains
only one point, so the claim trivially holds. If $n=2$, then the projective
space $\mathscr{S}_{n-1}$ consists of one projective line, i.e. $1$-dimensional
subspace, with $q+1$ points. Hence, it is trivially isomorphic to
$PG\left(n-1,q\right)$.

\begin{step}{17}
Graph $\varTheta\left(M_{n}\left(F_{q}\right)\right)$ satisfies properties ($i$)--($v$).
\end{step}

Let $G=\varTheta(M_n(F_q))$. Properties ($i$)--($iii$) follow from the proof of Propositions~\ref{prop:degree} and \ref{prop:same--graphs--same--rings} upon noticing that
\begin{equation}\label{eq:17-1}
V_k(G)=\set{(V,W) \in V(G)}{\dim_{F_q} V=k,\ \dim_{F_q} W=n-k}.
\end{equation}
Let $X$ be set of vertices of $G$, say $X=\set{(V_\alpha,W_\alpha)}{\alpha \in A}$. It is easy to check that
\begin{align}
\cl_t(X) &=\set{(V,W) \in V(G)}{V \subseteq \bigcup_{\alpha \in A} V_\alpha} \quad \textup{and} \label{eq:17-2}\\
\cl_s(X) &=\set{(V,W) \in V(G)}{\bigcap_{\alpha \in A} W_\alpha \subseteq W}. \label{eq:17-3}
\end{align}
Hence, $\cl_t(X)=\cl_t((V',W'))$ where $V'$ is the linear span of $\bigcup_{\alpha \in A} V_\alpha$ and $W'$ is any linear subspace of $F_q^n$ such that $\dim_{F_q} V'+\dim_{F_q} W'=n$.
Similarly, $\cl_s(X)=\cl_s((V'',W''))$ where $W''=\bigcap_{\alpha \in A} W_\alpha$ and $V''$ is any linear subspace of $F_q^n$ such that $\dim_{F_q} V''+\dim_{F_q} W''=n$. This proves property ($iv$).
To prove property ($v$), let $x=(V_x,W_x)$ and $y=(V_y,W_y)$ be elements of $V_k(G)$ for some $k$ and suppose $(V,W) \in \cl_t(x) \cap \cl_s(y) \cap V_k(G)$. Equalities~\eqref{eq:17-1}, \eqref{eq:17-2} and \eqref{eq:17-3} then imply $V \subseteq V_x$ and $\dim_{F_q} V=\dim_{F_q} V_x=k$, hence $V=V_x$. Similarly we get $W=W_y$. Therefore, we clearly have
$$\cl_t(x) \cap \cl_s(y) \cap V_k(G)=\{(V_x,W_y)\},$$
which proves ($v$).

\smallskip To finish the proof, we need to show that $G$ is isomorphic
to $\varTheta\left(M_{n}\left(F_{q}\right)\right)$. By Steps~15 and
16, there is a collineation $\varphi:\mathscr{S}_{n-1}\to PG\left(n-1,q\right)$.
By \cite[\S 9 Theorem~5 and \S 8 Theorem~17]{Ben}, this collineation
extends uniquely to a lattice isomorphism $\widehat{\varphi}$ between
the lattices of subspaces of the projective spaces involved. Observe that
for any $x\in V\left(G\right)$, the element $\widehat{\varphi}\left(S\left(\cl_{\textrm{t}}\left(x\right)\right)\right)$
corresponds to vector subspace of $F_{q}^{n}$ and will be henceforth
viewed as such. With this in mind, we can define a map $\psi:V\left(G\right)\to V\left(\varTheta\left(M_{n}\left(F_{q}\right)\right)\right)$
by the rule
\[
\psi\left(x\right)=\left(\widehat{\varphi}\left(S\left(\cl_{\textrm{t}}\left(x\right)\right)\right),\widehat{\varphi}\left(S\left(\cl_{\textrm{t}}\left(x^{\textrm{op}}\right)\right)\right)\right).
\]
Observe that $T\left(x\right)+T\left(x^{\textrm{op}}\right)=n$, so
$\dim \left(S\left(\cl_{\textrm{t}}\left(x\right)\right)\right)+\dim \left(S\left(\cl_{\textrm{t}}\left(x^{\textrm{op}}\right)\right)\right)=n-2$
and hence $\dim _{F_{q}}\left(\widehat{\varphi}\left(S\left(\cl_{\textrm{t}}\left(x\right)\right)\right)\right)+\dim _{F_{q}}\left(\widehat{\varphi}\left(S\left(\cl_{\textrm{t}}\left(x^{\textrm{op}}\right)\right)\right)\right)=n$.
This shows that $\psi\left(x\right)$ is indeed a vertex of $\varTheta\left(M_{n}\left(F_{q}\right)\right)$. 

\begin{step}{18}
$\psi$ is injective.
\end{step}

Suppose $\psi\left(x\right)=\psi\left(y\right)$. The injectivity
of $\widehat{\varphi}$ implies $S\left(\cl_{\textrm{t}}\left(x\right)\right)=S\left(\cl_{\textrm{t}}\left(y\right)\right)$
and $S\left(\cl_{\textrm{t}}\left(x^{\textrm{op}}\right)\right)=S\left(\cl_{\textrm{t}}\left(y^{\textrm{op}}\right)\right)$.
By Step~9, this further implies $\cl_{\textrm{t}}\left(x\right)=\cl_{\textrm{t}}\left(y\right)$
and $\cl_{\textrm{t}}\left(x^{\textrm{op}}\right)=\cl_{\textrm{t}}\left(y^{\textrm{op}}\right)$.
By Step~7, the last equality implies $N^{+}\left(x\right)=N^{+}\left(y\right)$
and hence $\cl_{\textrm{s}}\left(x\right)=\cl_{\textrm{s}}\left(y\right)$.
In addition, $T\left(x\right)=T\left(y\right)$ by Step~2. By property
\ref{enu:noncomm--sg5}, we obtain
\[
\left\{ x\right\} =\cl_{\textrm{t}}\left(x\right)\cap\cl_{\textrm{s}}\left(x\right)\cap V_{T\left(x\right)}\left(G\right)=\cl_{\textrm{t}}\left(y\right)\cap\cl_{\textrm{s}}\left(y\right)\cap V_{T\left(y\right)}\left(G\right)=\left\{ y\right\} .
\]
Hence, $\psi$ is injective.

\begin{step}{19}
$\psi$ is surjective.
\end{step}

Let $\left(V,W\right)$ be a vertex of $\varTheta\left(M_{n}\left(F_{q}\right)\right)$.
Then $\dim _{F_{q}}V+\dim _{F_{q}}W=n$. By property
\ref{enu:noncomm--sg4}, we have $\cl_{\textrm{t}}\left(\widehat{\varphi}^{-1}\left(V\right)\right)=\cl_{\textrm{t}}\left(v\right)$
for some $v\in V\left(G\right)$. Observe that, by Step~9, 
\begin{equation}
\widehat{\varphi}^{-1}\left(V\right)=S\left(\cl_{\textrm{t}}\left(v\right)\right)\label{eq:18}
\end{equation}
 which implies that $T\left(v\right)=\dim \widehat{\varphi}^{-1}\left(V\right)+1=\dim _{F_{q}}V$
since $\widehat{\varphi}$ is a lattice isomorphism. Similarly, $\cl_{\textrm{t}}\left(\widehat{\varphi}^{-1}\left(W\right)\right)=\cl_{\textrm{t}}\left(w\right)$
where $T\left(w\right)=\dim _{F_{q}}W$. In particular, $T\left(v\right)=T\left(w^{\textrm{op}}\right)$,
hence by property \ref{enu:noncomm--sg5}, there is a vertex $x\in V\left(G\right)$
such that $\left\{ x\right\} =\cl_{\textrm{t}}\left(v\right)\cap\cl_{\textrm{s}}\left(w^{\textrm{op}}\right)\cap V_{T\left(v\right)}\left(G\right)$.
By Step~1, we have $\cl_{\textrm{t}}\left(x\right)=\cl_{\textrm{t}}\left(v\right)$
and, by equality (\ref{eq:18}), $\widehat{\varphi}\left(S\left(\cl_{\textrm{t}}\left(x\right)\right)\right)=V$.
Similarly, $\cl_{\textrm{s}}\left(x\right)=\cl_{\textrm{s}}\left(w^{\textrm{op}}\right)$.
As in the proof of Step~18, this implies that $\cl_{\textrm{t}}\left(x^{\textrm{op}}\right)=\cl_{\textrm{t}}\left(w\right)$.
As before, this implies $\widehat{\varphi}\left(S\left(\cl_{\textrm{t}}\left(x^{\textrm{op}}\right)\right)\right)=W$.
We conclude that $\psi\left(x\right)=\left(V,W\right)$, so $\psi$
is surjective.

\begin{step}{20}
$\psi$ is a graph isomorphism.
\end{step}

Let $x\rightarrow y$ be an edge in $G$. In other words, $y\in N^{+}\left(x\right)$.
Then, by Step~7, $y\in\cl_{\textrm{t}}\left(x^{\textrm{op}}\right)$,
which implies $\cl_{\textrm{t}}\left(y\right)\subseteq\cl_{\textrm{t}}\left(x^{\textrm{op}}\right)$.
Since $\widehat{\varphi}$ is a lattice isomorphism, we deduce $\widehat{\varphi}\left(S\left(\cl_{\textrm{t}}\left(y\right)\right)\right)\subseteq\widehat{\varphi}\left(S\left(\cl_{\textrm{t}}\left(x^{\textrm{op}}\right)\right)\right)$.
Hence, there is an edge $\psi\left(x\right)\rightarrow\psi\left(y\right)$
in $\varTheta\left(M_{n}\left(F_{q}\right)\right)$. On the other
hand, if there is an edge $\psi\left(x\right)\rightarrow\psi\left(y\right)$
in $\varTheta\left(M_{n}\left(F_{q}\right)\right)$, then $\widehat{\varphi}\left(S\left(\cl_{\textrm{t}}\left(y\right)\right)\right)\subseteq\widehat{\varphi}\left(S\left(\cl_{\textrm{t}}\left(x^{\textrm{op}}\right)\right)\right)$.
The fact that $\widehat{\varphi}$ is a lattice isomorphism, and Step~9 imply $\cl_{\textrm{t}}\left(y\right)\subseteq\cl_{\textrm{t}}\left(x^{\textrm{op}}\right)$.
Hence, $y\in N^{+}\left(x\right)$ by Step~7, so there is an edge $x\rightarrow y$
in $G$. Since graphs involved have no multiple edges, this shows
that $\psi$ is a graph isomorphism.
\end{proof}

\begin{remark}
If $n=3$, then the graph $\varTheta\left(M_{n}\left(F_{q}\right)\right)$
still has the properties described in Theorem~\ref{prop:noncommutative--staircase},
but it is not the only graph that satisfies those properties, essentially because there exist non-Desarguesian projective planes. Every projective plane, Desarquesian or not, induces a graph with properties described in Theorem~\ref{prop:noncommutative--staircase}.
We simply take the vertices of the graph to be pairs of projective subspaces with projective dimensions summing up to $n-2$ and define edges via inclusions in the same way as they are defined in $\varTheta\left(M_{n}\left(F_{q}\right)\right)$. Nevertheless, graph $\varTheta\left(M_{3}\left(F_{q}\right)\right)$ also has all
of the properties deduced in Steps~1--14 and we will be using this fact
in the rest of the paper.
\end{remark}

\section{Structural connections between $M_{n}\left(F_{q}\right)$ and $\varTheta\left(M_{n}\left(F_{q}\right)\right)$}\label{sec:structure}

In this section we prove two strong connections between the ring $M_{n}\left(F_{q}\right)$ and the graph $\varTheta\left(M_{n}\left(F_{q}\right)\right)$, namely
\begin{enumerate}[leftmargin=*]
\item for $n \neq 2$, every graph automorphism of $\varTheta\left(M_{n}\left(F_{q}\right)\right)$ is induced by a ring automorphism of $M_{n}\left(F_{q}\right)$ (see Theorem~\ref{thm:automorphisms}), and
\item for $n\neq 1$, the graph structure of $\varTheta\left(M_{n}\left(F_{q}\right)\right)$ uniquely determines the ring structure of $M_{n}\left(F_{q}\right)$ (see Theorem~\ref{prop:matrix--same--graph--same--ring}).
\end{enumerate}

For an automorphism $\sigma$ of a field $F$ and a matrix $X=\left(x_{ij}\right)_{i,j}\in M_{m,n}\left(F\right)$,
written in the standard basis, we denote $X^{\sigma}=\left(\sigma\left(x_{ij}\right)\right)_{i,j}$.
For a subset $\mathcal{A}\subseteq M_{m,n}\left(F\right)$, we denote
$\mathcal{A}^{\sigma}=\left\{ X^{\sigma}:X\in\mathcal{A}\right\} $.
Recall that the group of automorphisms of a finite field $F_q$, where $q=p^m$, is a cyclic group of order $m$, generated by the automorphism $\tau(\lambda)=\lambda^p$.

\begin{theorem}\label{thm:automorphisms}
Let $n\neq2$ be a positive integer and $f$ a graph automorphism
of $\varTheta\left(M_{n}\left(F_{q}\right)\right)$. Then there is
an invertible matrix $A\in M_{n}\left(F_{q}\right)$ and an automorphism
$\sigma$ of $F_{q}$ such that $f\left(\left[X\right]\right)=\left[AX^{\sigma}A^{-1}\right]$
for all $X\in M_{n}\left(F_{q}\right)$.
\end{theorem}

\begin{proof}
Case $n=1$ is clear because $\varTheta\left(M_{1}\left(F_{q}\right)\right)$
has only one vertex with one loop. So, suppose $n\geq3$.
Let $\left(V_{1},W_{1}\right)$, $\left(V_{2},W_{2}\right)$, $\left(V_{1}',W_{1}'\right)$,
$\left(V_{2}',W_{2}'\right)$ be vertices of $\varTheta\left(M_{n}\left(F_{q}\right)\right)$
such that $f\left(\left(V_{1},W_{1}\right)\right)=\left(V_{1}',W_{1}'\right)$
and $f\left(\left(V_{2},W_{2}\right)\right)=\left(V_{2}',W_{2}'\right)$.
Observe that $V_{1}\subseteq V_{2}$ if and only if $V_{1}'\subseteq V_{2}'$,
and $V_{1}=V_{2}$ if and only if $V_{1}'=V_{2}'$. Indeed, by definition
of edges, inclusion $V_{1}\subseteq V_{2}$ is equivalent to $N^{-}\left(\left(V_{1},W_{1}\right)\right)\supseteq N^{-}\left(\left(V_{2},W_{2}\right)\right)$
which is further equivalent to $N^{-}\left(\left(V_{1}',W_{1}'\right)\right)\supseteq N^{-}\left(\left(V_{2}',W_{2}'\right)\right)$,
i.e. $V_{1}'\subseteq V_{2}'$. Equivalence of equalities is proved
similiarly. In the same manner, it can be shown that $W_{1}=W_{2}$
if and only if $W_{1}'=W_{2}'$. 

This allows us to define maps $\varphi,\varphi':L \to L$,
where $L$ is the lattice of subspaces of the projective space $PG\left(n-1,q\right)$,
by the condition
\begin{equation}
f\left(\left(V,W\right)\right)=\left(\varphi\left(V\right),\varphi'\left(W\right)\right)\label{eq:2}
\end{equation}
for all $\left(V,W\right)\in V\left(\varTheta\left(M_{n}\left(F_{q}\right)\right)\right)$.
It is clear that $\varphi$ is a bijective map. Also, observe that,
by the above, we have that $V_{1}\subseteq V_{2}$ if and only if
$\varphi\left(V_{1}\right)\subseteq\varphi\left(V_{2}\right)$. Hence,
$\varphi$ is a lattice isomorphism. By the Fundamental theorem of
projective geometry, $\varphi$ is induced by a bijective semilinear
map $\mathcal{S}$ on $F_{q}^{n}$. Let $\sigma$ be the corresponding automorphism
of the field $F_{q}$. We now prove that $\varphi=\varphi'$. We borrow
the notion of the opposite vertex from the proof of Theorem~\ref{prop:noncommutative--staircase}.
Observe that opposite vertex is defined in terms of neighbourhoods
and degrees, hence the relation of being opposite is preserved by
graph automorphism. From the proof of Proposition~\ref{prop:same--graphs--same--rings},
it follows that $\left(V,W\right)\in V_{k}\left(\varTheta\left(M_{n}\left(F_{q}\right)\right)\right)$,
where $k=\dim V$. Hence, it is easy to see that $\left(V,W\right)^{\textrm{op}}=\left(W,V\right)$.
Therefore,
\begin{align*}
\left(\varphi\left(V\right),\varphi'\left(W\right)\right) &= f\left(\left(V,W\right)\right)=f\left(\left(W,V\right)^{\textrm{op}}\right)=\\
&= f\left(\left(W,V\right)\right)^{\textrm{op}}=\left(\varphi\left(W\right),\varphi'\left(V\right)\right)^{\textrm{op}}=\left(\varphi'\left(V\right),\varphi\left(W\right)\right),
\end{align*}
which shows that $\varphi=\varphi'$.

Now, define a map $\mathcal{L}:F_{q}^{n} \to F_{q}^{n}$ by
$\mathcal{L}\left(x\right)=\mathcal{S}\left(x^{\sigma^{-1}}\right)$, where
$x \in F_q^n$ is written in the standard basis. It is easy to see that $\mathcal{L}$
is a bijective linear map. Let $A$ be its matrix in the standard
basis. Then $A$ is an invertible matrix and $\mathcal{S}\left(x\right)=Ax^{\sigma}$ for all $x \in F_q^n$.
Let $\left[X\right]=\left(V,W\right)$ be a vertex of $\varTheta\left(M_{n}\left(F_{q}\right)\right)$,
where $X\in M_{n}\left(F_{q}\right)$, $V=\im X$ and $W=\ker X$.
By the above, we have
\[
\varphi\left(V\right)=\left\{ \mathcal{S}\left(v\right):v\in V\right\} =\left\{ Av^{\sigma}:v\in V\right\} .
\]
On the other hand, $\im\left(AX^{\sigma}A^{-1}\right)=A\cdot\im\left(X^{\sigma}\right) = A \cdot\left(\im X\right)^{\sigma}$
since $\sigma$ is a field automorphism. Hence, $\varphi\left(V\right)=\im\left(AX^{\sigma}A^{-1}\right)$.
Similarly, $\varphi'\left(W\right)=\varphi\left(W\right)=\left\{ Aw^{\sigma}:w\in W\right\} $
and $\ker\left(AX^{\sigma}A^{-1}\right)=A\cdot\ker\left(X^{\sigma}\right)=A\cdot\left(\ker X\right)^{\sigma}=\varphi'\left(W\right)$.
By (\ref{eq:2}), we thus have
\[
f\left(\left[X\right]\right)=\left(\im\left(AX^{\sigma}A^{-1}\right),\ker\left(AX^{\sigma}A^{-1}\right)\right)=\left[AX^{\sigma}A^{-1}\right].
\]
\end{proof}

For $n=2$, Theorem~\ref{thm:automorphisms} fails. Observe that vertices in $\varTheta(M_2(F_q))$ distinct from $(0,F_q^2)$ and $(F_q^2,0)$ are all of the form $(V,W)$, where $V$ and $W$ are 1-dimensional subspaces of $F_q^2$. Let $\pi$ be any permutation of 1-dimensional subspaces of $F_q^2$. Then the map $\varphi$ defined by
\begin{align*}
\varphi((0,F_q^2)) &=(0,F_q^2),\\
\varphi((F_q^2,0)) &=(F_q^2,0),\\
\varphi((V,W)) &= (\pi(V),\pi(W)), \quad \dim_{F_q} V=\dim_{F_q} W=1,
\end{align*}
is a graph automorphism of $\varTheta(M_2(F_q))$. However, not all such maps are of the form as described in Theorem~\ref{thm:automorphisms}.

We remark that for the compressed zero-divisor graph $\Gamma_E$, the group of automorphisms of $\Gamma_E(M_2(F_q))$ was described in \cite{Ma-Wan-Zho}.
 
\begin{theorem}\label{prop:matrix--same--graph--same--ring}
If $\varTheta\left(R\right)\cong\varTheta\left(M_{n}\left(F_{q}\right)\right)$
where $n \neq 1$ and $R$ is a finite unital ring, then $R\cong M_{n}\left(F_{q}\right)$.
\end{theorem}

\begin{proof}
We will be using some notions introduced in the proof of Theorem~\ref{prop:noncommutative--staircase}
and some results obtained there. For each $x\in R$, define $U_{x}=\left\{ r\in R:\left[r\right]\in\cl_{\textrm{s}}\left(\left[x\right]\right)\right\} $.
Choose any representative $\widehat{x}\in R$ of the class $\left[x\right]^{\textrm{op}}$.
Since $\cl_{\textrm{s}}\left(\left[x\right]\right)=N^{-}\left(\left[\widehat{x}\right]\right)$
by Step~9 of the proof of Theorem~\ref{prop:noncommutative--staircase},
it follows that $U_{x}=\ann_{\ell}\left(\widehat{x}\right)$.
In particular, $U_{x}$ is a left ideal in $R$, hence $Rx\subseteq U_{x}$.
We will show that for every $a\in R$ such that $\left[a\right]\in V_{1}\left(\varTheta\left(R\right)\right)$,
we have in fact $U_{a}=Ra$. So, fix $a\in R$ such that $\left[a\right]\in V_{1}\left(\varTheta\left(R\right)\right)$.
Observe that, by Proposition~\ref{prop:degree}, we have $\left(V,W\right)\in V_{1}\left(\varTheta\left(M_{n}\left(F_{q}\right)\right)\right)$
if and only if $\dim _{F_{q}}V=1$. In addition, it is easy
to verify that for such a vertex, we have $N^{+}\left(\cl_{\textrm{t}}\left(\left(V,W\right)\right)\right)=\left\{ \left[0\right]\right\} $
since 
\begin{align*}
\cl_{\textrm{t}}\left(\left(V,W\right)\right) &= N^{+}\left(\left(V,W\right)^{\textrm{op}}\right)=N^{+}\left(\left(W,V\right)\right)=\\
&= \left\{ \left(X,Y\right):X\subseteq V,\dim _{F_{q}}Y=n-\dim _{F_{q}}X\right\} .
\end{align*}
Consequently, if $\left[r\right]\in V_{1}\left(\varTheta\left(R\right)\right)$,
then $N^{+}\left(\cl_{\textrm{t}}\left(\left[r\right]\right)\right)=\left\{ \left[0\right]\right\} \not\ni\left[a\right]$.
Hence, there exists $\left[r'\right]\in\cl_{\textrm{t}}\left(\left[r\right]\right)$
such that $r'a\neq0$. So, for each $\cl_{\textrm{t}}\left(\left[r\right]\right)\in\mathscr{P}_{1}$,
we can fix such an $r'$. Since $U_{a}$ is a left ideal, we have
$\left[r'a\right]\in\cl_{\textrm{s}}\left(\left[a\right]\right)$.
In a similar way we can show that $\left[r'a\right]\in\cl_{\textrm{t}}\left(\left[r\right]\right)$.
In particular, this implies that $\left[r'a\right]\in V_{1}\left(\varTheta\left(R\right)\right)$
by Step~2 of the proof of Theorem~\ref{prop:noncommutative--staircase}.
Hence, $\left[r'a\right]\in\cl_{\textrm{s}}\left(\left[a\right]\right)\cap V_{1}\left(\varTheta\left(R\right)\right)$.
Therefore, we can define a map
\[
\tau:\mathscr{P}_{1} \to \cl_{\textrm{s}}\left(\left[a\right]\right)\cap V_{1}\left(\varTheta\left(R\right)\right)
\]
by setting $\tau\left(\cl_{\textrm{t}}\left(\left[r\right]\right)\right)=\left[r'a\right]$, where $\cl_{\textrm{t}}\left(\left[r\right]\right) \in \mathscr{P}_{1}$.
Since $\left[r'a\right]\in\cl_{\textrm{t}}\left(\left[r\right]\right)$
and $\left[r'a\right]\in V_{1}\left(\varTheta\left(R\right)\right)$,
we have $\cl_{\textrm{t}}\left(\left[r'a\right]\right)=\cl_{\textrm{t}}\left(\left[r\right]\right)$ by Step 1 of the proof of Theorem~\ref{prop:noncommutative--staircase}.
This easily implies that $\tau$ is injective. Observe that, by the
proof of Theorem~\ref{prop:noncommutative--staircase}, we have $\left|\cl_{\textrm{s}}\left(\left[a\right]\right)\cap V_{1}\left(\varTheta\left(R\right)\right)\right|=\left|N^{-}\left(\left[a\right]^{\textrm{op}}\right)\cap V_{1}\left(\varTheta\left(R\right)\right)\right|=\binom{n}{1}_{q}=\left|\mathscr{P}_{1}\right|$.
Thus, $\tau$ is also surjective, which implies that $U_{a}\subseteq Ra$
since $\cl_{\textrm{s}}\left(\left[a\right]\right)=\left(\cl_{\textrm{s}}\left(\left[a\right]\right)\cap V_{1}\left(\varTheta\left(R\right)\right)\right)\cup\left\{ \left[0\right]\right\} $.
We conclude that $U_{a}=Ra$.

Next, we show that $R$ is a prime ring. Let $x,y\neq0$ be elements
of $R$, where $\left[y\right]\in V_{j}\left(\varTheta\left(R\right)\right)$,
$j>0$. Then, $\left|N^{-}\left(\left[y\right]\right)\cap V_{1}\left(\varTheta\left(R\right)\right)\right|=\binom{n}{1}_{q}\binom{n-j}{1}_{q}<\binom{n}{1}_{q}^{2}=\left|V_{1}\left(\varTheta\left(R\right)\right)\right|$,
hence there exists $r\in R$ such that $\left[r\right]\in V_{1}\left(\varTheta\left(R\right)\right)$
and $ry\neq0$. It follows that $\left[ry\right]\in V_{1}\left(\varTheta\left(R\right)\right)$.
So, by the above, $U_{ry}=Rry$. Similar argument as
before shows that $N^{-}\left(\cl_{\textrm{s}}\left(\left[ry\right]\right)\right)=\left\{ \left[0\right]\right\} \not\ni\left[x\right]$,
hence there exists $\left[s\right]\in\cl_{\textrm{s}}\left(\left[ry\right]\right)$
such that $xs\neq0$. This implies that $s\in U_{ry}=Rry$.
We conclude that $xRy\neq0$, which shows that $R$ is a prime ring.
As a finite ring, $R$ is also artinian, so it follows from Wedderburn-Artin
Theorem that $R$ is isomorphic to $M_{n'}\left(D\right)$ for some
division ring $D$ and nonnegative integer $n'$. Clearly, $D$ has to be finite, hence it is commutative
by Wedderburn's little theorem. Any finite commutative division ring is isomorphic
to some Galois field. Hence, $R\cong M_{n'}\left(F_{q'}\right)$ for
some prime power $q'$. By Proposition~\ref{prop:same--graphs--same--rings},
we have $n'=n$ and $q'=q$.
\end{proof}

We remark that a version of Theorem~\ref{prop:matrix--same--graph--same--ring} for the uncompressed zero-divisor graph $\Gamma(R)$ was proved in \cite{Akb-Moh-2} and generalized to matrix rings over commutative rings in \cite{Akb-Moh-3}.

As a corollary to Theorems~\ref{prop:noncommutative--staircase} and \ref{prop:matrix--same--graph--same--ring} we obtain a graph-theoretic characterization of finite semisimple rings.

\begin{theorem}\label{thm:semisimple-iff}
A finite unital ring $R$ is semisimple
if and only if 
\[ \varTheta\left(R\right)\cong\prod_{i=1}^{m}\varTheta\left(M_{n_{i}}\left(F_{q_{i}}\right)\right) \]
 for some positive integers $n_{i}$ and prime powers $q_{i}$.
\end{theorem}
 
\begin{proof}
If $R$ is semisimple, then it is a direct product of matrix rings
over finite division rings. Since every finite division ring is a
field, the conclusion follows from Proposition~\ref{prop:noncommutative--preserves--products}.

Conversely, suppose $\varTheta\left(R\right)\cong\prod_{i=1}^{m}\varTheta\left(M_{n_{i}}\left(F_{q_{i}}\right)\right)$.
By Proposition~\ref{prop:noncommutative--preserves--products}, we
have $\varTheta\left(R\right)\cong\varTheta\left(M_{n_{1}}\left(F_{q_{1}}\right)\right)\times\varTheta\left(\prod_{i=2}^{m}M_{n_{i}}\left(F_{q_{i}}\right)\right)$.
Proposition~\ref{prop:noncommutative--preproduct} implies that $R\cong R_{1}\times R_{1}'$,
where $\varTheta\left(R_{1}\right)\cong\varTheta\left(M_{n_{1}}\left(F_{q_{1}}\right)\right)$
and $\varTheta\left(R_{1}'\right)\cong\prod_{i=2}^{n}\varTheta\left(M_{n_{i}}\left(F_{q_{i}}\right)\right)$.
By induction we get $R\cong\prod_{i=1}^{n}R_{i}$, where $\varTheta\left(R_{i}\right)\cong\varTheta\left(M_{n_{i}}\left(F_{q_{i}}\right)\right)$.
If $n_{i}=1$ for some $i\in\left\{ 1,2,\ldots,m\right\} $, then
$\varTheta\left(R_{i}\right)$ has only two vertices, so every nonzero
element of $R_{i}$ is invertible. Thus, being finite, $R_{i}$ is
a field. If $n_{i}>1$ for some $i\in\left\{ 1,2,\ldots,m\right\} $,
then $R_{i}\cong M_{n_{i}}\left(F_{q_{i}}\right)$ by Theorem~\ref{prop:matrix--same--graph--same--ring}. Hence, $R$ is semisimple.
\end{proof}

\begin{theorem}\label{thm:iso-semisimple}
Let $R$ and $S$ be finite unital rings, where $S$ is semisimple
and has no nonzero homomorphic image isomorphic to a field. If $\varTheta\left(R\right)\cong\varTheta\left(S\right)$,
then $R\cong S$.
\end{theorem}

\begin{proof}
The assumptions imply $S\cong\prod_{i=1}^{m}M_{n_{i}}\left(F_{q_{i}}\right)$,
where $n_{i}>1$ for all $i\in\left\{ 1,2,\ldots,m\right\} $. Hence, the
conclusion follows from the proof of Theorem~\ref{thm:semisimple-iff}.
\end{proof}

\section{Infinite rings}\label{sec:inf}

In analogy with the commutative case (see \cite[Definition~6.1]{Dju-Jev-Sto}), the following is a possible extension of Definition~\ref{def:theta} to infinite rings.

\begin{definition}\label{def:inf}
For an arbitrary unital ring $R$, $\varTheta(R)$ is a graph whose vertex set is the set of equivalence classes of elements of $R$, where two elements $a,b \in R$ are equivalent if and only if $aR=bR$ and $Ra=Rb$. There is a directed edge $[x] \to [y]$ in $\varTheta(R)$ ($[x]$ and $[y]$ not necessarily distinct) if and only if $xy=0$.
\end{definition}

Proposition~\ref{prop:aux} and its left-hand sided version ensures that for finite rings (and even for artinian rings) Definitions~\ref{def:inf} is equivalent to Definition~\ref{def:inf}.

\section{Acknowledgements}
This research was supported by the Slovenian Research Agency, project number BI-BA/16-17-025.


\begin{thebibliography}{99}
\bibitem{And-Axt-Sti} D.F. Anderson, M.C. Axtell, J.A. Stickles, Zero-divisor graphs in commutative rings, In: Fontana M., Kabbaj SE., Olberding B., Swanson I. (eds), \emph{Commutative Algebra}, pp. 23--45, Springer, New York, NY, 2011.
\bibitem{Akb-Mai-Yas} S. Akbari, H.R. Maimani, S. Yassemi,
When a zero-divisor graph is planar or a complete r-partite graph, \emph{J. Algebra} 270 (2003), no. 1, 169--180. 
\bibitem{Akb-Moh} S. Akbari, A. Mohammadian, On the zero-divisor graph of a commutative ring, \emph{J. Algebra} 274 (2004), no. 2, 847--855.
\bibitem{Akb-Moh-2} S. Akbari, A. Mohammadian, Zero-divisor graphs of non-commutative rings, \emph{J. Algebra} 296 (2006), 462--479.
\bibitem{Akb-Moh-3} S. Akbari, A. Mohammadian, On zero-divisor graphs of finite rings, \emph{J. Algebra} 314 (2007), 168--184.
\bibitem{And-etal} D.F. Anderson, A. Frazier, A. Lauve, P.S. Livingston,
The zero-divisor graph of a commutative ring. II, \emph{Ideal theoretic methods in commutative algebra (Columbia, MO, 1999)}, 61--72,
Lecture Notes in Pure and Appl. Math., vol. 220, Dekker, New York, 2001.
\bibitem{And-LaG} D.F. Anderson, J.D. LaGrange, Commutative Boolean monoids, reduced rings, and the compressed zero-divisor graph, \emph{J. Pure Appl. Algebra} 216 (2012), no. 7, 1626--1636.
\bibitem{And-LaG-2} D.F. Anderson, J.D. LaGrange, Some remarks on the compressed zero-divisor graph, \emph{J. Algebra} 447 (2016), 297--321. 
\bibitem{And-Liv} D.F. Anderson, P.S. Livingston, The zero-divisor graph of a commutative ring, \emph{J. Algebra} 217 (1999), no. 2, 434--447.
%\bibitem{Ati-Mac} M.F. Atiyah, I.G. Macdonald: \emph{Introduction to commutative algebra}, Addison-Wesley publishing company, 1969.
\bibitem{Bec} I. Beck, Coloring of commutative rings, \emph{J. Algebra} 116 (1988), no. 1, 208--226.
\bibitem{Beh-Bey} M. Behboodi, R. Beyranvand: Strong Zero-Divisor Graphs
of Non-Commutative Rings, \emph{International Journal of Algebra} 2 (2008), no. 1, 25--44.
\bibitem{Ben} M.K. Bennett, \emph{Affine and projective geometry}, A Wiley-Interscience Publication, John Wiley \& Sons, Inc., New York, 1995.
\bibitem{Bon-Mur} J.A. Bondy, U.S.R. Murty: \emph{Graph theory with applications}, North-Holand, Amsterdam, 1976.
\bibitem{Boz-Pet} I. Bo\v zi\' c, Petrovi\' c: Zero-divisor graphs of matrices over commutative rings, \emph{Comm. Algebra} 37 (2009), 1186--1192.
\bibitem{Cas} R. Casse: \emph{Projective geometry: An introduction}, Oxford University Press, Oxford, 2006.
\bibitem{Coy-etal} J. Coykendall, S. Sather-Wagstaff, L. Sheppardson, S. Spiroff, On zero divisor graphs, \emph{Progress in commutative algebra 2}, 241--299, Walter de Gruyter, Berlin, 2012.
\bibitem{Dju-Jev-Obl-Sto} A.~\DJ uri\' c, S.~Jev\dj eni\' c, P. Oblak, N.~Stopar, The total zero-divisor graph of commutative rings, preprint, arXiv:1803.05628v1 [math.RA]
\bibitem{Dju-Jev-Sto} A. \DJ uri\' c, S. Jev\dj eni\' c, N. Stopar: Categorial properties of compressed zero-divisor graphs of finite commutative rings, preprint, arXiv:1807.11283v1 [math.RA].
\bibitem{Ham-Imr-Kla} R. Hammack, W. Imrich, S. Klavžar: \emph{Handbook of product graphs}, Second edition, CRC Press, Boca Raton FL, 2011.
\bibitem{Hil-Rhe} C.J. Hillar, D.L. Rhea: Automorphisms of finite abelian groups, \emph{Amer. Math. Monthly} 114 (2007), no. 10, 917--923. 
\bibitem{Hir} J.W.P. Hirschfeld: \emph{Projective geometries over finite fields}, The Clarendon Press, Oxford, 1979.
%\bibitem{Hun} T.W. Hungerford: On the structure of principal ideal rings, \emph{Pacific J. Math.} 25 (1968), no. 3, 543–547. 
%\bibitem{Kap} I. Kaplansky: Elementary divisors and modules, \emph{Trans. Amer. Math. Soc.} 66 (1949), 464--491.
\bibitem{Lam} T.Y. Lam: \emph{A First Course in Noncommutative Rings, Second Edition}, Springer Science + Business Media, New York, 2001.
%\bibitem{Lan} S. Lang: \emph{Algebra, Revised 3rd ed.}, Springer-Verlag, New York, 2002.
\bibitem{Ma-Wan-Zho} X. Ma, D. Wang, J. Zhou: Automorphisms of the zero-divisor graph over $2\times 2$ matrices, \emph{J. Korean Math. Soc.} 53 (2016), no. 3, 519--532.
\bibitem{Mul} S.B. Mulay: Cycles and symmetries of zero-divisors, \emph{Comm. Algebra} 30 (2002), no. 7, 3533--3558.
\bibitem{Red} S.P. Redmond: The Zero-Divisor Graph of a Non-Commutative Ring, \emph{Internat. J. Commutative Rings} 1 (2002), no. 4, 203--211.
\bibitem{Red-3} S.P. Redmond: An ideal-based zero-divisor graph of a commutative ring, \emph{Comm. Algebra} 31 (2003), no. 9, 4425--4443.
\bibitem{Red-2} S.P. Redmond: Structure in the zero-divisor graph of a noncommutative ring, \emph{Houston J. Math.} 30 (2004), no. 2, 345--355.
\bibitem{Smi} N.O. Smith, Planar zero-divisor graphs, \emph{Focus on commutative rings research}, 177--186, Nova Sci. Publ., New York, 2006. 
\bibitem{Spi-Wic} S. Spiroff, C. Wickham: A zero divisor graph determined by equivalence classes of zero divisors, \emph{Comm. Algebra} 39 (2011), no. 7, 2338--2348.
\bibitem{Sta} R.P. Stanley: \emph{Enumerative combinatorics, Volume 1}, Second edition, Cambridge University Press, Cambridge, 2012.
\end{thebibliography}
\end{document}